%% file: Main/main.tex
\documentclass[11pt]{article}
\usepackage{macros}
\usepackage{xurl}

\bibliography{references}

\usepackage{subfiles}

\title{
    \vspace{-1.2cm}
    Cusp Excursions, Lattice Points on Manifolds, and the Mizohata-Takeuchi Conjecture
}
\author{Inbo Gottlieb Fenves}
\date{June 22, 2026}

\begin{document}
\maketitle
\vspace{-1cm}
\begin{abstract}
    \scriptsize{We prove new logarithm laws for cusp excursions in spaces of lattices, and produce quantitative lower bounds for lattice points near submanifolds, using tools from dynamics and the geometry of numbers. As an application, we provide a new proof of power loss for the local Mizohata-Takeuchi conjecture as in \cite{CaiZha} with explicit error terms, as well as show that power loss is generic in $C^k$. The construction uses high-dimensional probabilistic estimates, but replaces the random orthogonal subspaces of \cite{CaiZha} with random unimodular lattices; this yields stronger bounds and provides a richer family of counterexamples.}
\end{abstract}

\setcounter{tocdepth}{1}

{\footnotesize
\begin{spacing}{0.2}
    \tableofcontents
\end{spacing}
}

\subfile{Sections/1-intro}

\subfile{Sections/2-reduction}

\subfile{Sections/3-latticegeometry}

\subfile{Sections/4-pointsnearmanifolds}

\subfile{Sections/5-maximalfunctions}

\subfile{Sections/6-proofofmt}

\subfile{Sections/7-genericity}

\appendix

\subfile{Sections/8-bookkeeping}

\subfile{Sections/9-cuspvolumes}

\nocite{*}
\printbibliography[ heading=bibintoc, title={References}]
\end{document}

%% file: Sections/1-intro.tex
\section{Introduction}

Fix $n \ge 2$ and let $\Sigma \subseteq \R^n$ be a compact $C^2$ hypersurface with surface measure $\dsigma$. We define the \textit{Fourier extension operator} $\Ec_\Sigma$ on functions $f \colon \Sigma \to \C$ by
\begin{equation*}
    \Ec_\Sigma f(x) = (f\dsigma)^\vee(x) = \int_{\Sigma} e^{2\pi i x \cdot y}f(y)\dsigma(y).
\end{equation*}
We also let $X$ denote the \textit{X-ray transform}, defined for $w \colon \R^n \to \C$ and an affine line $\ell \in \AGr_1(\R^n)$ by
\begin{equation*}
    Xw(\ell) = \int_\ell w.
\end{equation*}
The \textit{Mizohata-Takeuchi conjecture} predicts the behavior of $\Ec_\Sigma$ as a weighted $L^2$ operator.

\begin{conj}[Mizohata-Takeuchi]
    If $\Sigma$ is convex, then for all $f \in L^2(\Sigma)$ and measurable weights $w \colon \R^n \to [0,\infty)$,
    \begin{equation*}
        \norm{\Ec_\Sigma f}_{L^2(\R^n,w\dx)}^2 \lesssim \norm{X w}_{L^\infty(\AGr_1(\R^n))}\norm{f}_{L^2(\Sigma)}^2.
    \end{equation*}
\end{conj}

This conjecture was recently disproven by Cairo in the breakthrough paper \cite{Cai}, and it was later shown by Cairo-Zhang in \cite{CaiZha} that any surface may be perturbed so that a local form of Mizohata-Takeuchi fails with power loss; for $k=2$ this bound is sharp up to the endpoint in view of the results of Carbery-Iliopoulou-Wang in \cite{CarIliWan}. 

Our first result is a new proof of \cite{CaiZha}, with an explicit error term.

\begin{thm}\label{thm:LocalMTFailsalways}
    Let $\Sigma$ be a $C^k$-smooth compact hypersurface in $\R^n$. Then there exist arbitrarily small $C^k$ perturbations $\Sigma'$ of $\Sigma$ so that for all $R \gg 1$, there exist $f \in L^2(\Sigma')$ and weights $w$ supported on $\B_R^n$ for which
    \begin{equation*}
        R^{\frac{n-1}{n-1+k}}\norm{Xw}_{L^\infty(\AGr_1(\R^n))}\norm{f}_{L^2(\Sigma')}^2 \norm{\Ec_{\Sigma'} f}_{L^2(\R^n,w\dx)}^{-2} \lesssim \exp\left(O(\sqrt{\log R \log\log R})\right).
    \end{equation*}
\end{thm}

Here, $\B_R^n$ denotes the ball in $\R^n$ of radius $R$ about the origin. This improves on Cairo-Zhang, who obtain an error of the shape $\lesssim_\delta R^\delta$.

Our second result is that power loss in the Mizohata-Takeuchi conjecture is in fact \textit{generic} in the space of $C^k$-submanifolds. For simplicity, we consider the space of graphs of the form
\begin{equation*}
    \Sigma_\psi = \left\{ (\xi,\psi(\xi)) : \xi \in [0,1]^{n-1}\right\},
\end{equation*}
for some $\psi \in C^k([0,1]^{n-1})$.

\begin{thm}\label{thm:localMTgenericallyfails}
    Fix $\delta > 0$. Then the set of $\psi \in C^k$ for which there exist arbitrarily large $R \ge 1$, functions $f \in L^2(\Sigma_\psi)$ and weights $w \colon \B_R^n \to [0,\infty)$ satisfying
    \begin{equation*}
        \norm{\Ec_{\Sigma_\psi}f}_{L^2(\R^n,w\dx)}^2 \gtrsim R^{\frac{n-1}{n-1+k}}R^{-\delta}\norm{Xw}_{L^\infty}\norm{f}_{L^2(\Sigma_\psi)}^2
    \end{equation*}
    contains a dense $G_\delta$-set in the $C^k$-topology.
\end{thm}

Theorems \ref{thm:LocalMTFailsalways} and \ref{thm:localMTgenericallyfails} can be approached probabilistically, where the weights $w_R$ are chosen to approximate random lattices in $\R^N$ for some $N \gg n$. Cairo-Zhang use randomly distributed \textit{orthogonal} lattices\footnote{\cite{CaiZha} actually use the dual perspective of random orthogonal \textit{subspaces}, with a fixed lattice; moreover, they require an additional random dilation.} and high-dimensional incidence estimates for convex sets. We instead construct weights associated to random \textit{unimodular} lattices, and use moments of Siegel transforms and cusp estimates to establish our bounds. We establish the precise error term of Theorem \ref{thm:LocalMTFailsalways} in Appendices \ref{section:bookkeeping} and \ref{section:cuspvolumes}.

Along the way, we prove several results which may be of independent interest. The first is a lower bound for the number of lattice points near a submanifold, possibly contained in a low-dimensional hyperplane. Suppose $N \ge n$, $\iota \colon \R^n \hookrightarrow \R^N$ is an isometric linear embedding with transpose $\pi = {}^t\iota$, and $\Sigma \subseteq \R^n$ is an $m$-dimensional compact manifold.

Let $G = \SL_N(\R)$, $\Gamma = \SL_N(\Z)$. The quotient space $G/\Gamma$ is equipped with a finite volume form, which we normalize to be a probability measure $\mu$. We may identify $G/\Gamma$ with the space $\lattice$ of unimodular lattices in $\R^N$.

\begin{thm}\label{thm:countingnearmanifolds}
    Suppose that $N \ge 3$. Let $\Sigma \subseteq \B_1^n$ be an $m$-dimensional compact $C^k$ submanifold. Then for all $\epsilon > 0$ and all $R \gg_\epsilon 1$, the set
    \begin{equation*}
        \left\{ \vbf \in \pi(R^{-1/n}\Lambda \cap \B_1^N) : \dist(\vbf,\Sigma) < \epsilon R^{-\frac{k}{m+k(n-m)}} \right\}
    \end{equation*}
    contains $\gtrsim \epsilon^{n-m} R^{\frac{m}{m+k(n-m)}}$ points which are $\gtrsim R^{-\frac{1}{m+k(n-m)}}$-separated, for a set of $\Lambda \in \lattice$ of $\mu$-measure at least $\frac{1}{100}$.
\end{thm}

Note that this theorem holds even when $N=n$, and all implied constants can be made explicit. Moreover, the argument given here in fact shows that this probability may be made arbitrarily close to 1. A corresponding result is also true for $n=N=2$, $m=1$, with weaker constants.

The second involves cusp excursions of random lattices in $\lattice$ permitted to rotate around a positive codimension subspace. To state it, let $\alpha(\Lambda)$ denote the inverse of the smallest covolume of any dimension $1 \le p \le N-1$ sublattice $\Delta \le \Lambda$; this is the \textit{height function} of $\lattice$. Consider the subgroup $H = \SO(n) \times \Id_{N-n} \le \SL_N(\R)$. Choose some $\zbf = (z_1,\dots,z_N)$ with $z_i \ge z_{i+1}$ and $\sum_i z_i = 0$, and let $g^t = \exp(-t\zbf)$ be the corresponding one-parameter subgroup of $\SL_N(\R)$. The \textit{maximal height function} $\Max_\zbf^t \colon \lattice \to [0,\infty)$ is given by
\begin{equation*}
    \Max_\zbf^t(\Lambda) = \sup_{h \in H} \alpha(g^th\Lambda).
\end{equation*}
The function $\Max_\zbf^t$ measures how far into the cusp the lattice $\Lambda$ goes under the $g^t$-flow, when one is first allowed to rotate the lattice along the subspace $\R^n \subseteq \R^N$.

\begin{thm}\label{thm:maximalestimates}
    For a.e. $\Lambda \in \lattice$, we have
    \begin{equation*}
        \frac{\rho_\zbf}{N} \le \liminf_{t \to \infty} \frac{1}{t}\log \Max_\zbf^t(\Lambda) \le \limsup_{t \to \infty} \frac{1}{t}\log \Max_\zbf^t(\Lambda) \le \frac{\delta_\zbf}{N},
    \end{equation*}
    where
    \begin{equation*}
        \rho_\zbf = nz_1 - \sum_{i=1}^n z_i \quad \mathrm{and} \quad \delta_\zbf = \sum_{1 \le i < j \le n} (z_i-z_j).
    \end{equation*}
\end{thm}

\subsection{Additional directions}

There has been much recent work in dynamics aimed at understanding moments of transforms defined on moduli spaces: see, e.g., \cite{FaiHan}, \cite{KelYu}, \cite{Kim} for Siegel transforms on alternative spaces of lattices, and e.g., \cite{Vee}, \cite{AthCheMas} for Siegel-Veech transforms on strata of translation surfaces. Using the techniques of this paper, one can produce lower bounds of the same shape as Theorems \ref{thm:countingnearmanifolds} and \ref{thm:maximalestimates} in these other geometric contexts.

Cairo-Zhang prove a more general result in \cite[Theorem 5.3]{CaiZha}, where $\Sigma$ is allowed to be an arbitrary $m$-dimensional submanifold, and where weights are studied according to their $l$-plane transforms. These same results should follow from the methods of this paper, using a different choice of $\zbf$ as in Theorem \ref{thm:cuspestimatefinitetimeprobupper}.

Since the arguments in this paper are probabilistic (just like in \cite{CaiZha}), we can say nothing about the local Mizohata-Takeuchi conjecture for a \textit{fixed} surface $\Sigma$ (e.g., $\Sigma = \S^{n-1}$).

\subsection{Some notation}

Throughout this paper, there will be dimensional constants $n \le N$, a smoothness parameter $k \ge 2$, a perturbative parameter $\epsilon > 0$, and a scaling parameter $R \ge 1$.

We write $A \lesssim B$ or $A = O(B)$ if $A \le CB$ for some constant $C > 0$, and we write $A \asymp B$ if $A \lesssim B$ and $B \lesssim A$. The notation $A \ll B$ means that $A \le cB$ for some constant $c > 0$ which is additionally assumed to be quite small (thus, $\lesssim$ and $\ll$ formally have the same meaning). We write $A \lessapprox B$ if $A \lesssim_\delta R^\delta B$ for all $\delta > 0$. The notation $A = o(B)$ means that $A/B \to 0$ as the relevant parameters are extremized.

For auxiliary parameter(s) $par$, we will use subscripts in the asymptotic notation to denote dependency on $par$; for example, we have $A \lesssim_{par} B$ if $A \le C_{par} B$, where $C_{par}$ is a constant depending on $par$. We will allow the implicit constants in $\lesssim,\ll,\lessapprox,O(-),o(-)$ to depend on the ambient dimensions $n,N$ and the smoothness parameter $k$, but will be independent of the perturbative parameter $\epsilon$ and the scaling parameter $R$, unless otherwise specified. In Appendix \ref{section:bookkeeping}, we will study the explicit dependencies of these constants on $N$ to obtain the precise error bound stated in Theorem \ref{thm:LocalMTFailsalways}.

Throughout, we fix for each $N \ge 2$ a bump function $b_1 \in C_c^\infty(\R^N)$ which is radially symmetric, satisfies $\ind_{\B_1} \le b_1 \lesssim \ind_{\B_2}$, and $\hat{b}_1 \ge 0$. We write $b_R(x) = b_1(R^{-1}x)$, and we let $\phi_R = \hat{b}_R$.

Given $S \subseteq \R^n$ and $r > 0$, we write
\begin{equation*}
    \Nc_r(S) = \{x \in \R^n : \dist(x,S) < r\} \quad \mathrm{and} \quad rS = \{rs : s \in S\}.
\end{equation*}
If $\Lambda \subseteq \R^n$ is discrete (in particular, a lattice), we write $\delta_\Lambda = \sum_{\vbf \in \Lambda} \delta_\vbf$ for the Dirac distribution on $\Lambda$.

We use the notation $\B_r^m(x) = \{y \in \R^m : \dist(x,y) < r\}$ for the $r$-ball in $\R^m$ about a point $x$. When $x = 0$, we simply write $\B_r^m = \B_r^m(0)$, and moreover when the dimension is clear we write simply $\B_r$.

%% file: Sections/2-reduction.tex
\section{Reductions and a sketch}\label{section:reductions}

To start, note that after translation and restriction to a coordinate patch, we can (and always will) assume $\Sigma \subseteq \B_1^n$ without affecting any estimates.

We will first establish the inequality
\begin{equation*}
    R^{\frac{n-1}{N}}\norm{\Ec_{\Sigma'}f}_{L^2(\B_R,w\dx)}^2 \gtrsim_N R^{\frac{n-1}{n-1+k}}\norm{Xw}_{L^\infty(\AGr_1(\R^n))} \norm{f}_{L^2(\Sigma')}^2;
\end{equation*}
this will yield Theorem \ref{thm:LocalMTFailsalways} with error $\lessapprox 1$. In Appendix \ref{section:bookkeeping}, we will go through the specific dependencies in $N$ to obtain the stated error $\exp(O(\sqrt{\log R\log\log R}))$.

\subsection{A reduction}

Following Cairo-Zhang, we perform a technical reduction to pass from Theorem~\ref{thm:LocalMTFailsalways} to a statement at a fixed scale, which no longer involves the function $f$. There is no new content here, as all arguments are due to \cite{CaiZha}.

\begin{prop}\label{prop:oldmainproposition}
    Suppose $s_0 \colon [0,1]^{n-1} \to \R^n$ is a parameterization of a $C^k$ hypersurface $\Sigma_0$, so that $\norm{D^\gamma s_0}_{L^\infty} \lesssim 1$ for each $1 \le |\gamma| \le k$. Then for each $0 < \epsilon < 1$ and $R \gg_\epsilon 1$, there exist:
    \begin{enumerate}[label=(\arabic*)]
        \item Perturbations $\eta_R \colon \R^{n-1} \to \R^n$, supported on $[0,1]^{n-1}$, so that $\norm{D^\gamma \eta_R}_{L^\infty} \lesssim \epsilon$ for $0 \le |\gamma| \le k$, and associated submanifolds $\Sigma_R = (s_0 + \eta_R)([0,1]^{n-1})$;\footnote{As stated, this is trivial; however, in order to ensure that Proposition~\ref{prop:oldmainproposition}(2)(iii) holds, we will have to produce the perturbations $\eta_R$ probabilistically so as to satisfy a certain incidence estimate (see Lemma \ref{lemma:mainprobabilisticlemma}).}
        \item\label{weights} Weights $w_R \colon \B_R^n \to [0,\infty)$ satisfying:
        \begin{enumerate}[label=(\roman*)]
            \item $\norm{w_R}_{L^1} \lesssim R^{n-1}$.
            \item $\norm{X w_R}_{L^\infty} \lessapprox 1$.
            \item $\norm{\hat{w}_R}_{L^2(\Sigma_R)}^2 \gtrsim R^{\frac{n-1}{n-1+k}}R^{n-1}$.
        \end{enumerate}
    \end{enumerate}
\end{prop}

\begin{rmk}
    This result is the analogue of \cite[Proposition 5.3]{CaiZha}. In their results, extremely high-dimensional probabilistic estimates are needed to establish all of the properties above, and in particular they get $\lessapprox 1$-loss in both (ii) and (iii). We are able to remove this loss for (iii), although some loss in (ii) is inevitable using our methods (see Remark \ref{rmk:matchingbounds}). To prove Proposition \ref{prop:oldmainproposition}, we consider random weights $w_\Lambda$ associated to lattices in $\R^N$, for some $N \ge n$. The only place where we need to take $N$ large relative to $n$ is in (ii), as all other estimates hold when $N=n$.
\end{rmk}

\begin{lemma}[\cite{CaiZha}, Lemma 5.3]\label{lemma:errorreduction}
    Proposition~\ref{prop:oldmainproposition} implies Theorem~\ref{thm:LocalMTFailsalways}.
\end{lemma}

We will not give a full proof of Lemma~\ref{lemma:errorreduction}, but provide a sketch. Given $R \ge 1$, let us define the constant $C_\Sigma(R)>0$ to be the smallest constant for which
\begin{equation*}
    \norm{\Ec_\Sigma f}_{L^2(\B_R,w\dx)}^2 \le C_\Sigma(R)\norm{Xw}_{L^\infty}\norm{f}_{L^2(\Sigma)}^2
\end{equation*}
holds for all $f$ and all $w$. Moreover define $D_\Sigma(R) > 0$ to be the smallest constant for which
\begin{equation*}
    \norm{\hat{w}}_{L^2(\Sigma)}^2 \le D_\Sigma(R)\norm{Xw}_{L^\infty}\norm{w}_{L^1}
\end{equation*}
holds for all $w$. Then an argument involving duality and dyadic pigeonholing (\cite[Lemma 2.1]{CaiZha}) shows
\begin{equation*}
    D_\Sigma(R) \lesssim C_\Sigma(R) \lesssim \log(R)D_\Sigma(R).
\end{equation*}
Once this inequality is established, a patching argument (\cite[Lemma 5.4]{CaiZha}) allows us to pass from the single-scale statement of Proposition \ref{prop:oldmainproposition} to all scales $R$, with some small additional loss which may be absorbed into the larger error term.

\subsection{A sketch of the proof of Theorem \ref{thm:LocalMTFailsalways}, assuming Theorems \ref{thm:countingnearmanifolds} and \ref{thm:maximalestimates}}

In view of the previous reduction, it suffices to show a lower bound of the shape
\begin{equation*}
    \norm{\hat{w}}_{L^2(\Sigma)}^2 \gtrapprox R^{\frac{n-1}{n-1+k}}\norm{Xw}_{L^\infty}\norm{w}_{L^1},
\end{equation*}
for some weight $w$ supported on $\B_R^n$. Let us ignore for a moment the X-ray factor, and assume $n= N$. Take $h$ to be a suitably normalized approximation of $\delta_{R^{-1/n}\Lambda}$, where $\Lambda$ is a random (unimodular) lattice in $\R^n$, and let $w = \hat{h}^2$. Then $\norm{w}_{L^1}$ is roughly the number of elements of the dual lattice $(R^{-1/n}\Lambda)^\vee = R^{1/n}\Lambda^\vee$ lying in $\B_R^n$; this lattice has covolume $R$, so we should expect $\lesssim R^{n-1}$ points in this intersection.

On the other hand, $\norm{\hat{w}}_{L^2(\Sigma)}^2$ is roughly $R^{\dim \Sigma} = R^{n-1}$ times the number of points of $R^{-1/n}\Lambda$ which lie on $\Sigma$. By applying Theorem \ref{thm:countingnearmanifolds}, we can thicken our surface $\Sigma$ a bit to make this intersection quite large, on the order of $R^{\frac{n-1}{n-1+k}}$, while keeping it slim enough in the transversal directions that a subsurface $\Sigma_R$ which passes through all these lattice points remains $C^k$-close to $\Sigma$.

However, we have neglected the X-ray factor, and in fact staying in dimension $n$, this construction cannot produce a contradiction. The reason for this is that every lattice $L$ of covolume $R$ has a nonzero vector $\vbf \in L$ of length $\lesssim R^{1/n}$ (its \textit{systole}), and in particular $|Xw(\R \vbf)| \asymp R^{1-1/n}$.

To remedy this, let us instead embed $\R^n \subseteq \R^N$ for some $N \gg n$, and take $w$ to be sampled randomly from lattices in $\R^N$ instead. Since the X-ray transform still only ``sees'' the first $n$ directions, we only need to avoid the set of lattices which contains relatively short vectors lying very close to this high-codimension subspace. We can translate this into a maximal inequality for excursions of the associated lattices into the cusp; Theorem \ref{thm:maximalestimates} then shows that the probability such an event happens is small when $N$ is large.

\begin{rmk}
    We refer the reader to Carbery-Iliopoulou-Wang (\cite{CarIliWan}), Bennett-Guti\'errez-Nakamura-Oliveira (\cite{BenGutNakOli}), Carbery-Li-Pang-Yung (\cite{CarLiPanYun}), Cairo (\cite{Cai}), and Cairo-Zhang (\cite{CaiZha}) for recent results on the Mizohata-Takeuchi conjecture, as well as Stein (\cite{Ste}), Takeuchi (\cite{Tak1}, \cite{Tak2}), Mizohata (\cite{Miz}), and Barcelo-Ruiz-Vega (\cite{BarRuiVeg}) for classical references.
\end{rmk}

%% file: Sections/3-latticegeometry.tex
\section{Geometry of spaces of lattices}\label{section:latticegeometry}

To prove Proposition \ref{prop:oldmainproposition}, we use a probabilistic construction, with weights being chosen randomly over the space of lattices of an appropriate scale. To do so, we need some results from dynamics and the geometry of numbers. From now on, we fix some integer $N \ge n$; to prove Proposition \ref{prop:oldmainproposition} with error $\lesssim_\delta R^\delta$, we will eventually take $N \sim \delta^{-1}n$, but all results hold in general.\footnote{In \cite{CaiZha}, the authors require $N \gg \delta^{-1}n^2$, and so our high-dimensional probabilistic estimates are ``not as high-dimensional.''}

\subsection{Elementary facts on lattices}

Most definitions and facts in this subsection are standard, and can be found in e.g. \cite{Cas}. 

Recall the notation $G = \SL_N(\R)$, $\Gamma = \SL_N(\Z)$. Note that $G$ acts transitively on the space $\lattice$, and the stabilizer of $\Z^N \in \lattice$ is $\Gamma$, so we may identify $\lattice = G/\Gamma$. Since $\Gamma \le G$ is a lattice, the space $\lattice$ carries a canonical $G$-invariant probability measure, which we will denote by $\mu$ (see \cite{Sie}).

Given $\Lambda \in \lattice$, let $\Lambda^\vee \in \lattice$ be its \textit{dual lattice}. The subset $\Lambda_{\prim} \subseteq \Lambda$ refers to the set of \textit{primitive vectors}, i.e., vectors $\vbf \in \Lambda$ for which there does not exist $\wbf \in \Lambda$ and an integer $m \ge 2$ with $\vbf=m \wbf$.

We reserve the notation $\Lambda$ for a unimodular lattice, and write $L$ for a general lattice, with no condition on its covolume.

Let $\Kc_0^N$ denote the collection of all compact, convex, centrally symmetric bodies in $\R^N$.

\begin{thm}[\cite{VDC}]\label{thm:VDC}
    If $K \in \Kc_0^N$, then we have
    \begin{equation*}
        |\Lambda \cap K| \gtrsim \Vol(K)
    \end{equation*}
    uniformly in $\Lambda \in \lattice$.
\end{thm}

Let $K \in \Kc_0^N$, and define the successive minima $\lambda_1,\dots,\lambda_N$ associated to $K$ by
    \begin{equation*}
        \lambda_i(\Lambda) = \inf\left\{\lambda > 0 : \text{$\Lambda \cap K$ contains $i$ linearly independent vectors} \right\}.
    \end{equation*}

\begin{thm}[Minkowski's Second Theorem]\label{thm:minksecond}
    If $K \in \Kc_0^N$ has associated successive minima $\lambda_1,\dots,\lambda_N$, then for any lattice $L$ (of any covolume), we have
    \begin{equation*}
        \prod_{i=1}^N \lambda_i(L) \asymp \frac{\Covol(L)}{\Vol(K)}.
    \end{equation*}
\end{thm}

\begin{thm}[\cite{Hen}]\label{thm:succminest}
    For all $K \in \Kc_0^N$ and any lattice $L$, we have
    \begin{equation*}
        |L \cap K| \lesssim \prod_{i=1}^N \left( \frac{1}{\lambda_i(L)} + 1 \right).
    \end{equation*}
\end{thm}

\subsection{Moments of Siegel transforms}

Suppose $f \in C_c(\R^N)$; we define its \textit{Siegel transform} $\hat{f}^{\SV} \colon \lattice \to \C$ via the formula
\begin{equation*}
    \hat{f}^{\SV}(\Lambda) = \sum_{\vbf \in \Lambda_{\prim}} f(\vbf).
\end{equation*}

\textit{Siegel's mean value theorem} states that the Siegel transform is $L^1 \to L^1$-bounded.

\begin{thm}[\cite{Sie}]\label{thm:Siegelsformula}
    If $f \in C_c(\R^N)$, then
    \begin{equation*}
        \int_{\lattice} \hat{f}^{\SV}(\Lambda)\dmu(\Lambda) = \frac{1}{\zeta(N)}\int_{\R^N} f \dx.
    \end{equation*}
\end{thm}

\begin{rmk}
    In view of Siegel's mean value theorem, the map $f \mapsto \int_{\lattice} \hat{f}^{\SV}\dmu$ defines a locally finite Borel measure on $\R^N$, and thus it is straightforward to extend results on compactly supported continuous functions to general bounded, compactly supported measurable functions. We will implicitly use these arguments in what follows, and in particular refer to $\hat{\ind}_B^{\SV}$ for bounded measurable sets $B$, which is the (primitive) lattice point counting function.
\end{rmk}

\begin{cor}
    If $B \subseteq \R^N$ is a bounded set, then
    \begin{equation*}
        \E_{\lattice}\left[ |\Lambda_{\prim} \cap B| \right] = \frac{1}{\zeta(N)}\Vol(B).
    \end{equation*}
\end{cor}

We will also need control of the tails of the measure $\mu$. In the case $N \ge 3$, there is an explicit estimate for the second moment of the Siegel transform, known as \textit{Rogers' second moment formula}.\footnote{A formula for the second moment of the Siegel transform exists in the $N=2$ case, but it is not as well-behaved, and working with it directly requires nontrivial integral estimates in some $KAK$-decomposition; see, e.g., \cite{FaiHan}. Rogers incorrectly states that his formula holds for $N=2$; this is noted in \cite{Sch}.}

\begin{thm}[\cite{Rog}, Lemma 1]\label{thm:Rogersformula}
    Suppose $N \ge 3$, and $f \in C_c(\R^N)$. Then
    \begin{equation*}
        \int_{\lattice} (\hat{f}^{\SV})^2\dmu = \frac{1}{\zeta(N)^2}\left(\,\int_{\R^N} f\dx \right)^2 + \frac{1}{\zeta(N)}\int_{\R^N}f(x)^2\dx + \frac{1}{\zeta(N)}\int_{\R^N}f(x)f(-x)\dx.
    \end{equation*}
\end{thm}

\begin{cor}\label{cor:variancebound}
    Suppose $N \ge 3$, and $B \subseteq \R^N$ is a bounded set. Let $X = \hat{\ind}_B^{\SV}$. Then
    \begin{equation*}
        \Var(X) \le 2\E[X].
    \end{equation*}
\end{cor}
\begin{proof}
    We compute
    \begin{align*}
        \Var(X) &= \E[X^2] - \E[X]^2 \\
        &= \frac{1}{\zeta(N)^2}\Vol(B)^2 + \frac{1}{\zeta(N)}\left( \Vol(B) + \Vol(B \cap -B)\right) - \frac{1}{\zeta(N)^2}\Vol(B)^2 \\
        &\le 2\zeta(N)^{-1}\Vol(B) \\
        &= 2\E[X].
    \end{align*}
\end{proof}

\subsection{Height functions, thick-thin decompositions, and cusps}\label{subsection:heightfunctions}

The results of this subsection and the following are adapted from Eskin-Margulis-Mozes in \cite{EskMarMoz}.

For $1 \le p \le N-1$, we define the \textit{height functions} $\alpha_p \colon \lattice \to [0,\infty)$ by
\begin{equation*}
    \alpha_p(\Lambda) = \sup\left\{ \frac{1}{\Covol_V(\Lambda \cap V)} : V \in \Gr(p,N) \right\},
\end{equation*}
where by convention $\Covol_V(\Lambda \cap V) = \infty$ if $\Lambda \cap V$ is not a lattice in $V$. We additionally define $\alpha = \max_p \alpha_p$.

The (inverse of the) function $\alpha_1$ is typically called the \textit{systole}, and records the shortest nonzero vector of a lattice. The following result is standard, and follows from a packing argument.

\begin{lemma}\label{lemma:sysbounds}
    The inequality
    \begin{equation*}
        |\Lambda \cap \B_r^N(x)| \lesssim (r\alpha_1(\Lambda))^N
    \end{equation*}
    holds uniformly in $r \gtrsim 1$, $\Lambda \in \lattice$, and $x \in \R^N$.
\end{lemma}

For each $\epsilon > 0$, we have an associated \textit{thick-thin decomposition} $\lattice = \lattice^{\thick,\epsilon} \uplus \lattice^{\thin,\epsilon}$ of the space of lattices via
\begin{equation*}
    \lattice^{\thick,\epsilon} = \left\{\Lambda \in \lattice : \alpha(\Lambda) \le \epsilon^{-1} \right\} \quad \mathrm{and} \quad \lattice^{\thin,\epsilon} = \left\{ \Lambda \in \lattice : \alpha(\Lambda) > \epsilon^{-1} \right\}.
\end{equation*}

\begin{thm}[Mahler's Compactness Theorem]
    The sets $(\lattice^{\thick,\epsilon})_{\epsilon > 0}$ form a compact exhaustion of $\lattice$.
\end{thm}

\begin{figure}[H]
    \centering
    \includegraphics[scale=0.7]{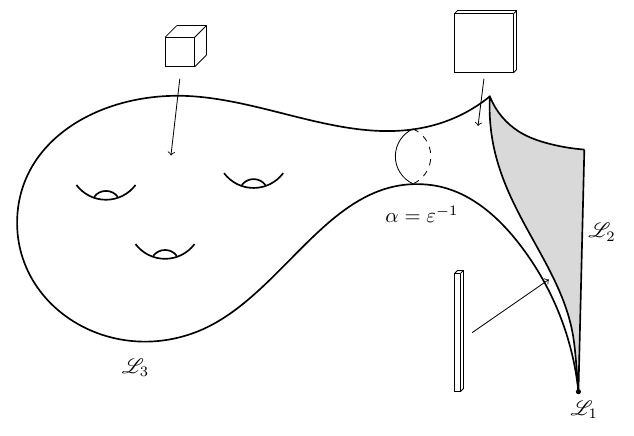}
    \caption{The Thick-Thin Decomposition of $\Ls_3$}
    \label{fig:placeholder}
\end{figure}

To prove Theorem \ref{thm:LocalMTFailsalways} with $\lessapprox 1$-error bounds, we will require the following (semi)-quantitative version of Mahler compactness.

\begin{thm}\label{thm:cuspvolume}
    For all $N \ge 3$ and all $\epsilon > 0$, we have
    \begin{equation*}
        \mu(\lattice^{\thin,\epsilon}) \asymp \epsilon^N.
    \end{equation*}
\end{thm}
\begin{proof}
    We present here a simple proof which does not recover the optimal leading constant: a slightly more sophisticated argument is given in Appendix \ref{section:cuspvolumes} to get the error bound stated in Theorem \ref{thm:LocalMTFailsalways}.
    
    Consider the Iwasawa decomposition $G = KAU$, where $K = \SO(N)$, $A \le G$ is the subgroup of positive diagonal matrices, and $U \le G$ is the subgroup of strictly upper triangular matrices (see e.g. \cite{Mor}). The Haar measure on $G$ decomposes as
    \begin{equation*}
        \int f(g)\dg = \int_K \int_A \int_U f(kau)\prod_{i < j} \frac{a_i}{a_j} \du\da\dk.
    \end{equation*}
    We define a \textit{Siegel set} $S \subseteq G$ by $S = KA_2U_1$, where
    \begin{equation*}
        A_2 = \{a \in A : a_i \le 2a_{i+1}\} \quad \mathrm{and} \quad U_1 = \left\{ u \in U : |u_{ij}| \le 1 \right\}.
    \end{equation*}
    It is well-known that $S\Gamma = G$, and in fact $S$ is a coarse fundamental domain for $\Gamma$ (see e.g. \cite[Ch. 7]{Mor}). Thus, it suffices to estimate the volume of
    \begin{equation*}
        S(\epsilon) = \left\{ s \in S : \alpha(s\Z^N) > \epsilon^{-1} \right\}.
    \end{equation*}
    As $\alpha$ is $K$-invariant (and $K$ is compact), the $K$-integral is irrelevant. Moreover, since the diagonal entries of $a \in A_2$ are weakly decreasing and have product 1, the set $U_1$ is weakly stable under the $A$-action, i.e., $A_2U_1 \subseteq U_{O(1)}A_2$. The condition numbers of $u \in U_{O(1)}$ are uniformly bounded, and so in fact
    \begin{equation*}
        \int_K \int_A \int_U \ind_{S(\epsilon)}(kau)\prod_{i < j} \frac{a_i}{a_j} \du\da\dk \asymp \int_{\{a \in A_2 : a_1 \cdots a_p < \epsilon \text{ for some $p$}\}}\prod_{i < j} \frac{a_i}{a_j} \da.
    \end{equation*}
    The Haar measure on $A$ is given by $da = \prod_{i=1}^{N-1} \frac{\da_i}{a_i}$, so we are reduced to estimating
    \begin{equation*}
        \sum_{p=1}^{N-1}\int_{\{a \in A_2 : a_1 \cdots a_p < \epsilon\}}\left( \prod_{i < j} \frac{a_i}{a_j} \right)\left( \prod_{i=1}^{N-1} a_i \right)^{-1}\da_1 \cdots \da_{N-1}.
    \end{equation*}
    We perform the substitution $b_i = a_i/a_{i+1}$. The Jacobian of this change of variables is $\propto \frac{1}{a_1}$, and moreover $\prod_{i < j} \frac{a_i}{a_j} = \prod_{i=1}^{N-1}b_i^{i(N-i)}$. Since $\prod_{i=1}^N a_i = 1$, we see that
    \begin{equation*}
        \int_{\{a \in A_2 : a_1 \cdots a_p < \epsilon\}}\left( \prod_{i < j} \frac{a_i}{a_j} \right)\left( \prod_{i=1}^{N-1} a_i \right)^{-1}\da_1 \cdots \da_{N-1} \lesssim \int_{B_p(\epsilon)}\prod_{i=1}^{N-1}b_i^{c_i-1}\db_1 \cdots \db_{N-1},
    \end{equation*}
    where $B_p(\epsilon) = \{b \in [0,2]^{N-1} : a_1 \cdots a_p \le \epsilon\}$, and $c_i = i(N-i)$. We transform this into a product over the $b_i$ via
    \begin{equation*}
        a_1 \cdots a_k = \prod_{i=1}^{N-1} b_i^{q_i}, \quad \mathrm{where} \quad q_i = \left\{ \begin{array}{cc}
            i(N-p)/N, & i \le p \\
            p(N-i)/N, & i > p
        \end{array} \right.
    \end{equation*}
    Note that $\frac{c_i}{q_i} \ge N$ uniformly in $i$, which implies $\int_{B_p(\epsilon)} \prod_i b_i^{c_i-1}\db \lesssim \epsilon^N$, as desired.

    The corresponding lower bound $\gtrsim \epsilon^N$ follows from taking $p=1$ as above and explicitly integrating.
\end{proof}

\begin{rmk}\label{rmk:suboptimalcuspbound}
    By explicitly performing the integration, one sees that this argument yields
    \begin{equation*}
        \mu\left(\lattice^{\thin,\epsilon} \right) \le C^{N^3}\epsilon^N
    \end{equation*}
    for some $C \ge 1$. Theorem \ref{thm:cuspvolume} is essentially contained in \cite{EskMarMoz}, proof of Lemma 3.10, and can be proven in (at least) two alternate ways beyond the ones presented here; one can use equidistribution of random geodesics on $G/\Gamma$ alongside the logarithm laws of \cite{KleMar1}, or work in local period coordinates in the style of \cite{Doz}.
\end{rmk}

\subsection{Cusp excursions and counting}

In this subsection, we will relate the height $\alpha(\Lambda)$ of a lattice $\Lambda$ to the quantity $|\Lambda \cap \B_1^N|$.

\begin{prop}\label{prop:counttocusp}
    If $|\Lambda \cap \B_1^N| \ge T$, then $\Lambda \in \lattice^{\thin,O(T^{-1})}$.
\end{prop}
\begin{proof}
    Let $\Delta \subseteq \Lambda$ be the sublattice spanned by $\Lambda \cap \B_1^N$, and write $d = \dim \Delta$. Consider\footnote{By slight abuse of notation.} the ball $\B_1^d = \B_1^N \cap \Span_\R(\Delta)$, and let $\lambda_1,\dots,\lambda_d$ be the successive minima associated to $\B_1^d$. Minkowski's second theorem implies
    \begin{equation*}
        \prod_{i=1}^d \lambda_i(\Delta) \asymp \Covol(\Delta).
    \end{equation*}
    Note that $\lambda_i(\Delta) \le 1$ for each $i$ by assumption. Therefore, by Theorem \ref{thm:succminest} we have
    \begin{equation*}
        T \le |\Lambda \cap \B_1^d| \lesssim  \left(\, \prod_{i=1}^d \lambda_i(\Delta) \right)^{-1}.
    \end{equation*}
    This implies
    \begin{equation*}
        \Covol(\Delta) \lesssim T^{-1},
    \end{equation*}
    which in turn implies $\alpha(\Lambda) \gtrsim T$, as desired.
\end{proof}

Conversely, we have the following.

\begin{prop}\label{prop:cusptocount}
    If $\Lambda \in \lattice^{\thin,\epsilon}$, then $|\Lambda \cap \B_1^N| \gtrsim \epsilon^{-1}$.
\end{prop}
\begin{proof}
    Suppose $\alpha_d(\Lambda) = \alpha(\Lambda)$, and let $\Delta \subseteq \Lambda$ be the $d$-dimensional sublattice realizing this supremum. Let $\B_1^d = \B_1^N \cap \Span_\R(\Delta)$; then this is a $d$-dimensional convex symmetric set of volume $\asymp 1$, and so by Theorem \ref{thm:VDC} we have
    \begin{align*}
        |\Delta \cap \B_1^d| \gtrsim \Covol(\Delta)^{-1},
    \end{align*}
    as desired.
\end{proof}

Most important for us is the following immediate corollary.

\begin{cor}\label{cor:height=count}
    We have the estimate
    \begin{equation*}
        \alpha(\Lambda) \asymp |\Lambda \cap \B_1^N|,
    \end{equation*}
    uniformly in $\Lambda \in \lattice$.
\end{cor}

As an aside, we obtain the following integrability condition for Siegel transforms of compactly supported functions on $\R^N$.

\begin{cor}\label{cor:lpestimate}
    If $f \in C_c(\R^N)$, then the non-primitive Siegel transform $\hat{f}^{\SV'} \in L^p(\lattice)$ if and only if $1 \le p < N$.
\end{cor}
\begin{proof}
    It suffices to show the integrability for $f = \ind_{\B_r}$, so that $\hat{f}^{\SV'}(\Lambda) = |\Lambda \cap \B_r| - 1$. By a layer-cake decomposition, we obtain
    \begin{align*}
        \int_{\lattice} |\Lambda \cap \B_r^N|^p\dmu(\Lambda)  &= p\int_0^\infty t^{p-1} \mu(\{\Lambda \in \lattice : |\Lambda \cap \B_r^N| \ge t\}) \dt \\
        &\asymp_{p} 1 + \int_1^\infty t^{p-1}\mu(\lattice^{\thin,O_r(t^{-1})}) \dt \\
        &\asymp_r \int_1^\infty t^{p-1-N}\dt,
    \end{align*}
    which is finite if and only if $p < N$.
\end{proof}

\begin{rmk}\label{rmk:integrability}
    The same proof shows that $\hat{f}^{\SV'} \in L^\Phi(\lattice)$ if and only if $\int_1^\infty \Phi'(t)\frac{\dt}{t^N} < \infty$. In the case where $p \in \{1,\dots,N-1\}$ is an integer, an explicit formula for these moments is established by Rogers in \cite{Rog} (with finiteness established in general in \cite{Sch}). Of course, $f$ only needs to be bounded and compactly supported.
\end{rmk}

%% file: Sections/4-pointsnearmanifolds.tex
\section{Points near manifolds}\label{section:ptsnearmflds}

The goal of this section is to prove Theorem \ref{thm:countingnearmanifolds}.

Assume that $\Sigma \subseteq \B_1^n$ is an $m$-dimensional compact $C^k$ submanifold. Using the implicit function theorem, we may reduce to the case where $\Sigma$ is the graph of a function
\begin{equation*}
    \Sigma_\psi = \left\{ (\xi,\psi(\xi)) : \xi \in [0,1]^m \right\},
\end{equation*}
for some $\psi \in C^k([0,1]^m;\R^{n-m})$ with $\norm{\psi}_{C^k} \lesssim 1$.\footnote{The choice of norm on $C^k$ is unimportant; for concreteness, take $\norm{f}_{C^k} = \max_{0 \le |\alpha| \le k} \norm{D^\alpha f}_{L^\infty}$.}

Let us set
\begin{equation*}
    \beta = \frac{1}{m+k(n-m)} \quad \mathrm{and} \quad q = \frac{m}{m+k(n-m)}.
\end{equation*}
Given $R \ge 1$, we subdivide $[0,1]^m$ into $\sim R^{q}$ cubes $\{I_\theta\}_{\theta \in \Theta}$ of side lengths $\sim R^{-\beta}$. Fix some $0 < \epsilon < 1$. For each $\theta \in \Theta$, we define the associated \textit{thickening}
\begin{equation*}
    \Sigma_\psi(\epsilon,\theta) = \left\{ (\xi,\psi(\xi) + y) : \xi \in I_\theta,\ y \in [0,\epsilon R^{-k\beta})^{n-m} \right\}.
\end{equation*}
Up to shrinking $\Theta$ by a multiplicative factor of at most $3^m$, we may assume that the sets $\Sigma_\psi(\epsilon,\theta)$ are all $R^{-\beta}$-separated.

\subsection{Counting points in sectors}

Fix some $N \ge \max\{n,3\}$, and set $C_0 = (\frac{3^m\zeta(N)}{\omega_{N-n}})^{1/q}$, where $\omega_l = \Vol(\B_1^l)$ is the volume of the $l$-dimensional unit ball. Let $0 < \rho_0 \le 1$ be an additional parameter.

\begin{thm}\label{thm:countingpointsnge3}
    For all $0 < \epsilon < \frac{1}{10}$ and all $R \ge C_0\rho_0^{-(N-n)/q}\epsilon^{-\frac{n-m}{q}}$, we have
    \begin{equation*}
        \P\left( \left| \left\{ \theta \in \Theta : R^{-1/N}\Lambda \cap \left( \Sigma_\psi(\epsilon,\theta) \times \B_{\rho_0}^{N-n} \right) \ne \emptyset \right\} \right| \ge \frac{\omega_{N-n}\rho_0^{N-n}}{6 \cdot 3^m\zeta(N)}\epsilon^{n-m} R^{q} \right) \ge \frac{1}{100}.
    \end{equation*}
\end{thm}
\begin{proof}
    To start, let us compute
    \begin{equation*}
        \Vol\left( \Sigma_\psi(\epsilon,\theta) \times \B_{\rho_0}^{N-n}\right) = R^{-q} \epsilon^{n-m}R^{-(n-m)k\beta}\omega_{N-n}\rho_0^{N-n} = \omega_{N-n}\rho_0^{N-n}\epsilon^{n-m} R^{-1}.
    \end{equation*}
    Thus, the dilates $T_\theta = R^{1/N}\left( \Sigma_\psi(\epsilon,\theta) \times \B_{\rho_0}^{N-n}\right)$ are each of volume $\omega_{N-n}\rho_0^{N-n}\epsilon^{n-m}$. Let us define the functions
    \begin{equation*}
        f_\theta = \ind_{T_\theta}, \quad f = \sum_{\theta \in \Theta} f_\theta, \quad X_\theta = \hat{f}_{\theta}^{\SV}, \quad X = \sum_{\theta \in \Theta} X_\theta.
    \end{equation*}
    Set $E = \E[X_\theta]$. By Siegel's mean value theorem, this is well-defined independently of $\theta$, and in fact
    \begin{equation*}
        E = \frac{\omega_{N-n}\rho_0^{N-n}}{\zeta(N)}\epsilon^{n-m}.
    \end{equation*}
    Rogers' second moment formula implies
    \begin{equation*}
        \E[X_\theta^2] \le E^2 + 2E.
    \end{equation*}
    Thus, by the second moment method we have
    \begin{equation*}
        \P(X_\theta > 0) \ge \frac{E^2}{\E[X_\theta^2]} \ge \frac{E}{2+E}.
    \end{equation*}
    Now, set $Y_\theta = \ind_{X_\theta > 0}$ and $Y = \sum_\theta Y_\theta$. Then
    \begin{equation*}
        \E[Y] \ge \frac{E}{2+E} |\Theta|.
    \end{equation*}
    Paley-Zygmund implies
    \begin{align*}
        \P\left(Y \ge \frac{1}{2}\E[Y] \right) \ge \frac{1}{4}\frac{\E[Y]^2}{\E[Y^2]} \ge \frac{1}{4}\frac{\E[Y]^2}{\E[X^2]} &\ge \frac{1}{4}\frac{|\Theta|^2\frac{E^2}{(2+E)^2}}{\E[X]^2 + \Var(X)} \\
        &= \frac{1}{4(2+E)^2}\frac{|\Theta|^2E^2}{|\Theta|^2E^2 + \Var(X)}.
    \end{align*}
    By Corollary \ref{cor:variancebound}, we have $\Var(X) \le 2|\Theta|E$, and so
    \begin{equation*}
        \frac{|\Theta|^2E^2}{|\Theta|^2E^2 + \Var(X)} \ge \frac{1}{1 + \frac{2}{|\Theta|E}}.
    \end{equation*}
    Therefore, so long as $|\Theta|E \ge 1$, we have
    \begin{equation*}
        \P\left( Y \ge \frac{1}{2}\E[Y] \right) \ge \frac{1}{12(2+E)^2}.
    \end{equation*}
    The bound $|\Theta|E \ge 1$ holds so long as $R \ge C_0\rho_0^{-(N-n)/q}\epsilon^{-(n-m)/q}$.
    
    Now, a straightforward computation shows $E = \frac{\omega_{N-n}\rho_0^{N-n}}{\zeta(N)}\epsilon^{n-m} \le 0.6$ uniformly in all parameters. Thus we have
    \begin{equation*}
        \P\left( Y \ge \frac{1}{2}\E[Y] \right) \ge \frac{1}{100}.
    \end{equation*}
    Since $\E[Y] \ge \frac{E}{2+E}|\Theta|$, we have
    \begin{equation*}
        \P\left( Y \ge \frac{1}{2}\E[Y] \right) \le \P\left( Y \ge \frac{E}{6}|\Theta| \right) \le \P\left( Y \ge \frac{\omega_{N-n}\rho_0^{N-n}}{6 \cdot 3^m\zeta(N)}\epsilon^{n-m} R^{q} \right).
    \end{equation*}
    Finally, note that
    \begin{equation*}
        Y = \left| \left\{ \theta \in \Theta : \Lambda_{\prim} \cap T_\theta \ne \emptyset \right\} \right|,
    \end{equation*}
    which completes the proof.
\end{proof}

\subsection{Proof of Theorem \ref{thm:countingnearmanifolds}}

\begin{proof}[Proof of Theorem \ref{thm:countingnearmanifolds}]
    This follows from Theorem \ref{thm:countingpointsnge3}, after noting that the set of lattices which do not map injectively under $\R^n \times \R^{N-n} \to \R^n$ has positive codimension, and in particular has measure zero. Observe that the separation condition on the set $\Theta$ guarantees that these projected points remain $R^{-\beta}$-separated even under the projection.
\end{proof}

\begin{rmk}\label{rmk:highprob}
    The estimate $\P \ge \frac{1}{100}$ is far from optimal, and the above argument in fact gives a lower bound of the shape $\P \ge \frac{1}{4+o(1)}$; here, the constant $\frac{1}{4} = \frac{1}{2^2}$ ``comes from'' the fact $\mathrm{Aut}(\Lambda) = \{\pm 1\}$ for a $\mu$-random $\Lambda$. In fact, with slightly more effort one can show $\P = 1-o(1)$.

    The strategy for proving Theorem~\ref{thm:countingnearmanifolds} is similar to the one employed by Athreya-Margulis in \cite{AthMar}. For us, the explicit formulation of Rogers' second moment formula is used both to control individual events, and as an independence condition on the events $|\Lambda \cap T_\theta| \ne 0$.
\end{rmk}

\begin{rmk}
    The question of counting lattice points near manifolds is of central importance in diophantine approximation, and has been studied from multiple perspectives. We refer the reader to Kleinbock-Margulis (\cite{KleMar}), Bernik-Kleinbock-Margulis (\cite{BerKleMar}), Iosevich-Sawyer-Seeger (\cite{IosSawSee}), Iosevich-Taylor (\cite{IosTay}), Beresnevich (\cite{Ber}), Huang (\cite{Hua}), Schindler-Srivastava-Technau (\cite{SchSriTec}), Srivastava (\cite{Sri}), and Smith (\cite{Smi}) for an (incomplete) list of references including a mix of number-theoretic, harmonic-analytic, and dynamical approaches.
\end{rmk}

%% file: Sections/5-maximalfunctions.tex
\section{Maximal function estimates}\label{section:maximalfunctionestimates}

In this section, we study logarithm laws for escape rates of lattices which are allowed to rotate in an adversarial way along positive-codimension subspaces, with the goal of proving Theorem \ref{thm:maximalestimates}, as well as Theorem \ref{thm:cuspestimatefinitetimeprobupper}. 

To start, we identify $\R^n$ with its inclusion into $\R^N$ in the first $n$ coordinates:
\begin{equation*}
    \R^n = \{(x_1,\dots,x_n,0,\dots,0) : x_i \in \R \}.
\end{equation*}
Let $H \le G$ denote the subgroup of $K = \SO(N)$ which stabilizes the subspace $(\R^n)^\perp$. This is an $\frac{1}{2}n(n-1)$-dimensional compact subgroup, and in matrix coordinates we have
\begin{equation*}
    H = \begin{bmatrix}
        \SO(n) & \\
        & \Id_{N-n}
    \end{bmatrix}.
\end{equation*}
Let $\afrak = \{\zbf \in \R^N : \sum_i z_i = 0\}$. We identify $\afrak$ with a Cartan subalgebra of $\slinear_N(\R) = \Lie(G)$, and let $\afrak_+ = \{\zbf \in \afrak : z_1 \ge \cdots \ge z_N\}$ be the associated Weyl chamber.

Suppose now that $\zbf \in \afrak_+$, and consider the one-parameter subgroup $g^t = \exp(-t\zbf)$. We will be interested in the maximal operator $\Max_\zbf^t \colon \lattice \to [0,\infty)$ defined by
\begin{equation*}
    \Max_\zbf^t(\Lambda) = \sup_{h \in H} \alpha(g^th\Lambda).
\end{equation*}
We moreover define constants $\delta_\zbf,\rho_\zbf > 0$ via the formulas
\begin{equation*}
    \delta_\zbf = \sum_{1 \le i < j \le n} (z_i-z_j) \quad \mathrm{and} \quad \rho_\zbf = nz_1 - \sum_{1 \le i \le n} z_i.
\end{equation*}

\begin{rmk}\label{rmk:matchingbounds}
    The function $\Max_\zbf^t$, as well as the constants $\delta_\zbf,\rho_\zbf$, implicitly depend on $n$. The identity $\delta_\zbf = \rho_\zbf$ holds if and only if $z_i = z_j$ whenever $i,j \ge 2$; this is the case we consider for proving Proposition \ref{prop:oldmainproposition}.
    
    Note that $\rho_\zbf > 0$ except in trivial cases. This shows that we cannot obtain $\lesssim 1$ for Proposition \ref{prop:oldmainproposition}(2)(ii) using this method. It would be interesting to find the correct almost sure asymptotic for $\limsup_t \frac{1}{t} \log \Max_\zbf^t$, for a general $\zbf$ (this limit exists a.s.; see Proposition \ref{prop:asloglaw}).

    Both $\delta_\zbf$ and $\rho_\zbf$ can be viewed as the topological entropy of the $g^t$-flow when restricted to certain invariant submanifolds of $\lattice$, as in \cite[Lemma 6.2]{EinKat}. See also e.g., \cite{Mor1}, \cite{KadKleLinMar}, \cite{EinLinMicVen} for the relationship between escape of mass and entropy.
\end{rmk}

The proof of Theorem \ref{thm:maximalestimates} follows in two parts; we first establish the upper bound by constructing an appropriate net of ellipsoids and then applying a Borel-Cantelli argument, and then prove the corresponding lower bound deterministically. In view of the results of Corollary \ref{cor:height=count}, we may equivalently study the quantity
\begin{equation}\label{eq:maxtoheight}
    \Max_\zbf^t(\Lambda) \asymp \sup_{h \in H} |g^th\Lambda \cap \B_1^N| = \sup_{h \in H} |\Lambda \cap hg^{-t}\B_1|.
\end{equation}

\begin{rmk}
    The connection between Theorem \ref{thm:cuspestimatefinitetimeprobupper} and the Mizohata-Takeuchi conjecture lies in the observation that $g^{-t}\B_1$ is an ellipsoid in $\R^N$ of dimensions $e^{z_1t} \times \cdots \times e^{z_Nt}$, and the $H$-action simply rotates this ellipse inside the subspace $\R^n$; thus, for an appropriate choice of $\zbf$, the operator $\Max_\zbf^t$ is essentially the X-ray transform of a certain weight on $\R^n$.
\end{rmk}

\subsection{Constructing a good net}

Temporarily, we work in $\R^n$. Let $\Fc$ denote the family of all ellipsoids in $\R^n$ of dimensions $e^{z_1t} \times \cdots \times e^{z_nt}$; then we may write $\Fc = \SO(n)F_0$, where $F_0$ is the standard axis-aligned ellipsoid of these dimensions. We say a finite subset $\{F_\theta\}_{\theta \in \Theta} \subseteq \Fc$ is a \textit{net} if, for all $F \in \Fc$, there exists some $F_\theta$ with $F \subseteq CF_\theta$ for some $C = O(1)$. Our goal is to establish an upper bound for $|\Theta|$.

\begin{lemma}\label{lemma:netsize}
    For all $t \ge 1$, we may choose $|\Theta| \lesssim e^{\delta_\zbf t}$.
\end{lemma}
\begin{proof}
    Let $\sortho_n(\R)$ be the Lie algebra of $\SO(n)$, and define $\gamma^t = \exp(-t(z_1,\dots,z_n))$. Then $F_0 = \gamma^{-t}\B_1^n$. Note that, for $\Omega \in \sortho_n(\R)$, we have
    \begin{equation*}
        \exp(\Omega)F_0 = \exp(\Omega)\gamma^{-t}\B_1 = \gamma^{-t}(\gamma^t\exp(\Omega)\gamma^{-t})\B_1 = \gamma^{-t}(\exp(\gamma^{t}\Omega \gamma^{-t}))\B_1.
    \end{equation*}
    Thus, a matrix $\Omega \in \sortho_n(\R)$ satisfies $\exp(\Omega)F_0 \subseteq CF_0$ if and only if $\gamma^t\Omega \gamma^{-t}$ has bounded entries. We have
    \begin{equation*}
        (\gamma^t\Omega\gamma^{-t})_{ij} = e^{(z_j-z_i)t}\Omega_{ij},
    \end{equation*}
    and so we must consider the set
    \begin{equation*}
        P = \left\{ \Omega \in \sortho_n(\R) : |\Omega_{ij}| \lesssim e^{(z_j - z_i)t} \right\}.
    \end{equation*}
    Since $\Omega$ is skew-symmetric and $\zbf$ is nonincreasing, we have
    \begin{equation*}
        \Vol(P) \asymp \prod_{1 \le i < j \le n} e^{(z_j-z_i)t} = e^{-\delta_\zbf t}.
    \end{equation*}
    Thus, we require $\lesssim e^{\delta_\zbf t}$ patches of this form to cover $\SO(n)$, so $|\Theta| \lesssim e^{\delta_\zbf t}$.
\end{proof}

\subsection{Establishing an upper bound}

The primary goal of this subsection is to establish the following probabilistic estimate.

\begin{thm}\label{thm:cuspestimatefinitetimeprobupper}
    For all $t \gg 1$, we have
    \begin{equation*}
        \P\left( \Max_\zbf^t(\Lambda) \ge Te^{\delta_\zbf t/N} \right) \lesssim T^{-N}.
    \end{equation*}
\end{thm}
\begin{proof}
    Fix a scale $t \gg 1$. In view of \eqref{eq:maxtoheight}, we have
    \begin{equation*}
        \P\left( \Max_\zbf^t(\Lambda) \ge A \right) \lesssim \P\left(\, \sup_{h \in H} |g^th\Lambda \cap \B_1^N| \gtrsim A \right).
    \end{equation*}
    Let us set $E_0 = g^{-t}\B_1^N$; then $E_0$ is an ellipsoid of dimensions $e^{z_1 t} \times \cdots \times e^{z_Nt}$ whose $\R^n$-slice is of the dimensions for which Lemma \ref{lemma:netsize} applies (note that any net for dimension $n$ will also be a net for $HE_0$, perhaps after changing the constant $C$). Thus, we can find a set $\{E_\theta\}_{\theta \in \Theta}$ of $\lesssim e^{\delta_\zbf t}$ ellipsoids of this form for which each $hE_0$ lies in $CE_\theta$ for some $\theta \in \Theta$ and $C = O(1)$. By the union bound we have
    \begin{equation*}
        \P\left(\, \sup_{h \in H} |g^th\Lambda \cap \B_1^N| \gtrsim A \right) \le \sum_{\theta \in \Theta} \P\left( |\Lambda \cap CE_\theta| \gtrsim A \right).
    \end{equation*}
    Now, there exists $g_\theta \in G$ for which $g_\theta CE_\theta = \B_C^N$, and since $\mu$ is $G$-invariant, we obtain
    \begin{equation*}
        \sum_{\theta \in \Theta} \P\left( |\Lambda \cap CE_\theta| \gtrsim A \right) = |\Theta|\P(|\Lambda \cap \B_C^N| \gtrsim A).
    \end{equation*}
    But now $\P(|\Lambda \cap \B_C^N| \ge A) \asymp \mu\bigl(\lattice^{\thin,O(A^{-1})}\bigr) \asymp A^{-N}$ by Corollary \ref{cor:height=count} and Theorem \ref{thm:cuspvolume}, and so taking $A = Te^{\delta_\zbf t/N}$ and applying Lemma \ref{lemma:netsize} yields
    \begin{equation*}
        \P\left( \Max_\zbf^t(\Lambda) \ge Te^{\delta_\zbf t/N} \right) \lesssim |\Theta|(Te^{\delta_\zbf t/N})^{-N} \lesssim T^{-N},
    \end{equation*}
    as desired.
\end{proof}

\begin{cor}\label{cor:BCupperbound}
    For a.e. $\Lambda \in \lattice$, we have
    \begin{equation*}
        \limsup_{t \to \infty} \frac{1}{t}\log \Max_\zbf^t(\Lambda) \le \frac{\delta_\zbf}{N}.
    \end{equation*}
\end{cor}
\begin{proof}
    By Theorem \ref{thm:cuspestimatefinitetimeprobupper}, we have
    \begin{equation*}
        \int_1^\infty \P \left( \Max_\zbf^t(\Lambda) \ge te^{\delta_\zbf t/N} \right) \dt \lesssim \int_1^\infty t^{-N} \dt < \infty.
    \end{equation*}
    By regularity of $\Max_\zbf^t$, we apply Borel-Cantelli to obtain
    \begin{equation*}
        \P\left( \Max_\zbf^t(\Lambda) \ge te^{\delta_\zbf t/N}\ \text{infinitely often} \right) = 0.
    \end{equation*}
    In particular, taking logarithms yields
    \begin{equation*}
        \limsup_{t \to \infty} \frac{1}{t}\log \Max_\zbf^t(\Lambda) \le \frac{\delta_\zbf}{N}
    \end{equation*}
    almost surely.
\end{proof}

\subsection{The corresponding lower bound}

In this section, we establish the deterministic lower bound.

\begin{lemma}\label{lemma:cuspestimatefinitetimeproblower}
    For all $\Lambda \in \lattice$ and all $t \gg 1$, we have
    \begin{equation*}
        \Max_\zbf^t(\Lambda) \gtrsim e^{\rho_\zbf t/N}.
    \end{equation*}
\end{lemma}
\begin{proof}
    Let us consider the region $D_M \subseteq \R^N$ given by
    \begin{equation*}
        D_M = \left\{ x \in \R^N : \norm{\proj_{\R^n}(x)} \lesssim M^{-1}e^{z_1t},\, |x_i| \lesssim M^{-1}e^{z_it}\ \text{for $i > n$} \right\}.
    \end{equation*}
    The set $D_M$ is an appropriate $\R^N$-thickening of $\B_{M^{-1}e^{z_1t}}^n \subseteq \R^n$. We compute
    \begin{equation*}
        \Vol(D_M) \asymp M^{-N}\exp\left({nz_1t + \sum_{i > n}z_it}\right) = M^{-N}e^{\rho_\zbf t}.
    \end{equation*}
    Thus, when $M \ll e^{\rho_\zbf t/N}$, we have $\Vol(D_M) \gg 1$, which by Minkowski's theorem implies that \textit{every} $\Lambda \in \lattice$ contains a nontrivial $v \in D_M \cap \Lambda$. But then we see that $v,2v,\dots,Mv$ all lie in a tube $hg^{-t}\B_{O(1)}^N$ for some $h \in H$, and so by Corollary \ref{cor:height=count} we obtain
    \begin{equation*}
        \alpha(g^th^{-1}\Lambda) \gtrsim e^{\rho_\zbf t/N}.
    \end{equation*}
\end{proof}

\begin{cor}\label{cor:BClowerbound}
    For all $\Lambda \in \lattice$ we have
    \begin{equation*}
        \liminf_{t \to \infty} \frac{1}{t}\log \Max_\zbf^t(\Lambda) \ge \frac{\rho_\zbf}{N}.
    \end{equation*}
\end{cor}

We have thus completed the proof of Theorem \ref{thm:maximalestimates}.

\begin{proof}[Proof of Theorem \ref{thm:maximalestimates}]
    This follows immediately from combining Corollaries \ref{cor:BCupperbound} and \ref{cor:BClowerbound}.
\end{proof}

\subsection{Logarithm laws}

Note that we can obtain more refined estimates on the behavior of $\Max_\zbf^t$ than Theorem \ref{thm:maximalestimates} using Theorem \ref{thm:cuspestimatefinitetimeprobupper}. In particular, we can show the following.

\begin{cor}\label{cor:loglaws}
    For a.e. $\Lambda \in \lattice$, we have
    \begin{equation*}
        \limsup_{t \to \infty} \frac{\log \Max_\zbf^t(\Lambda) - \delta_\zbf t/N}{\log t} \le \frac{1}{N}.
    \end{equation*}
\end{cor}
\begin{proof}
    Take $T = t^{1/N+\epsilon}$ in Theorem~\ref{thm:cuspestimatefinitetimeprobupper} and send $\epsilon \to 0$.
\end{proof}

In the case $n=1$, so that $\delta_\zbf = 0$, this recovers the upper bound of \cite{KleMar1}. We remark that the value $\limsup_{t \to \infty} \frac{1}{t}\log \Max_\zbf^t(\Lambda)$ depends on $\zbf$ whenever $n \ge 2$, in contrast to the $n=1$ case.

As an aside, let us show that for a.e. $\Lambda \in \lattice$, the value $\limsup_{t \to \infty} \log \Max_\zbf^t(\Lambda)$ is a fixed constant.

\begin{prop}\label{prop:asloglaw}
    Suppose $N > n$. Then there exists some $\rho_\zbf \le v_\zbf \le \delta_{\zbf}$ for which
    \begin{equation*}
        \limsup_{t \to \infty} \frac{1}{t}\log \Max_\zbf^t(\Lambda) = \frac{v_\zbf}{N}.
    \end{equation*}
    for a.e. $\Lambda$.
\end{prop}
\begin{proof}
    Since $N > n$, the centralizer $Z = C_G(AH)$ of $AH$ is noncompact, since it contains for example the one-parameter subgroup $a^s = \exp(s \diag(1,...,1,-(N-1)))$.
    By the Howe-Moore theorem (see e.g. \cite[Theorem 2.2.6]{Zim}), the action $Z \acts \lattice$ is ergodic.
    
    We claim the function $f \colon \lattice \to [0,\infty)$, $\Lambda \mapsto \limsup_t \frac{1}{t}\log \Max_\zbf^t(\Lambda)$ (which is clearly measurable) is $Z$-invariant. To see this, suppose $z \in Z$ and $\Lambda \in \lattice$. Then
    \begin{equation*}
        \Max_\zbf^t(z\Lambda) = \sup_{h \in H} \alpha(g^thz\Lambda) = \sup_{h \in H}\alpha(zg^th\Lambda).
    \end{equation*}
    Now, the condition number of $z$ is some fixed constant, and so we have
    \begin{equation*}
        \sup_{h \in H} \alpha(zg^th\Lambda) \asymp_z \alpha(g^th\Lambda),
    \end{equation*}
    so that
    \begin{equation*}
        f(z\Lambda) = \limsup_{t \to \infty} \frac{\log \Max_\zbf^t(w\Lambda)}{t} = \limsup_{t \to \infty} \frac{\log \Max_\zbf^t(\Lambda) + O_z(1)}{t} = f(\Lambda)
    \end{equation*}
    By ergodicity of the $Z$-action, it follows that $f$ is $\mu$-a.e. constant.
\end{proof}

The case $N=n$ is easier.

\begin{lemma}
    If $N=n$, then
    \begin{equation*}
        \lim_{t \to \infty} \frac{1}{t}\log \Max_\zbf^t(\Lambda) = \sum_{1 \le i \le N} \max(z_i,0).
    \end{equation*}
    for every lattice $\Lambda \in \lattice$.
\end{lemma}
\begin{proof}
    Recall $g^t = \exp(-t\zbf)$. Let $P \le G$ be the subgroup of upper triangular matrices. Then we have a decomposition $G = KP$, so that each $\Lambda \in \lattice$ has in its $K$-orbit a unique lattice $p \Z^N$ for some $p \in P$. Let $i_0 \ge 1$ be the largest index for which $z_i > 0$, and consider the subgroup $\Delta_0 \le \Z^N$ defined by $\Delta_0 = \Z \ebf_{1} + \cdots + \Z \ebf_{i_0}$. Then $\Delta = p\Delta_0 \le p\Z^N$ is a well-defined subgroup of $p\Z^N$, and $\Covol(\Delta) = \prod_{i = 1}^{i_0} p_{ii}$. Thus, we have
    \begin{equation*}
        \Covol(g^t\Delta) = \prod_{i=1}^{i_0} e^{-z_it}p_{ii}.
    \end{equation*}
    Now, we have $p_{ii} \asymp_\Lambda 1$, so
    \begin{equation*}
        \log \Covol(g^t\Delta) = -t\sum_{i=1}^{i_0} z_i + O_\Lambda(1).
    \end{equation*}
    Therefore, we obtain
    \begin{equation*}
        \liminf_{t \to \infty} \frac{1}{t}\log \Max_\zbf^t(\Lambda) \ge \sum_{i=1}^{i_0}z_i.
    \end{equation*}
    On the other hand, it is clear from the action of $g^t$ on $\R^N$ that
    \begin{equation*}
        \limsup_{t \to \infty} \frac{1}{t}\log \Max_\zbf^t(\Lambda) \le \sum_{i=1}^{i_0} z_i,
    \end{equation*}
    and so we are done.
\end{proof}

\begin{rmk}
    When $N=n$, we verify that
    \begin{equation*}
        \frac{\rho_\zbf}{N} = z_1 \le \sum_{1 \le i \le N} \max(z_i,0) \le \frac{1}{N}\sum_{i=1}^N \left( (N-i)z_i + \sum_{j=1}^i z_j \right) = \frac{\delta_\zbf}{N}.
    \end{equation*}
\end{rmk}

\subsection{Some additional remarks}

The maximal functions $\Max_\zbf^t$ can be compared to \textit{averaging operators} $\Avg^t$ of the form
\begin{equation*}
    \Avg^tf(\Lambda) = \int_K f(g^tk\Lambda)\dk.
\end{equation*}
These operators have been studied extensively in the context of quantitative nondivergence results, with the ``system of integral inequalities'' argument pioneered by Eskin-Margulis-Mozes in \cite{EskMarMoz} alongside the use of \textit{Margulis functions} (which can be constructed via suitable modifications of $\alpha$), and have had a tremendous influence on dynamics since. See \cite{EskMoz} for a survey article on Margulis functions, as well as \cite{EskMirMoh} for landmark results in the translation surface case. For us, the case $H=K$ is uninteresting, as then one can freely rotate any lattice $\Lambda$ so that it has a vector pointing exactly in the direction of strongest contraction under $g^t$, so that $\Max_\zbf^t(\Lambda) \gtrsim e^{z_1t} = e^{\rho_\zbf t/N}$. In particular, one cannot hope to extract bounds for $\Max$ from bounds on $\Avg$ without additional effort.

The averaging operators have also been studied in a more general context beginning with Eskin-McMullen in \cite{EskMcM}, where they prove equidistribution and counting-type results using averages of the shape
\begin{equation*}
    \int_{gH/(\Gamma \cap H)} f(h)\dh,
\end{equation*}
for $H \le G$ so that $G/H$ is \textit{affine symmetric}, i.e., so that $H$ is the fixed point set of an involution. In particular, see \cite[Theorem 3.1]{EskMcM}.

%% file: Sections/6-proofofmt.tex
\section{Power loss for local Mizohata-Takeuchi}

In this section, we apply Theorems \ref{thm:countingnearmanifolds} and \ref{thm:maximalestimates} to prove Proposition \ref{prop:oldmainproposition}, and consequently Theorem \ref{thm:LocalMTFailsalways}. Throughout, we assume $N \ge 3$, and let $\iota \colon \R^n \hookrightarrow \R^N$ be the standard embedding in the first $n$ coordinates, with adjoint $\pi \colon \R^N \to \R^n$. We also set $V = \iota \R^n$.

We in fact prove the following stronger version of Proposition \ref{prop:mainproposition}, which we will apply to obtain genericity results in Section \ref{section:genericity}. Let $0 < \rho_0 \le 1$ be a constant whose value is allowed to depend on $n,N$ but not $\epsilon,R$; this additional parameter will be pinned down in Lemma \ref{lemma:propertyiii}.

Fix now some $\psi \in C^k([0,1]^{n-1})$, some $0 < \epsilon < \frac{1}{10}$, and some $R \ge C_0\rho_0^{-(N-n)/q}\epsilon^{-1/q}$ as defined in Theorem \ref{thm:countingpointsnge3}.

\begin{prop}\label{prop:mainproposition}
    There exists a set $\Gc_{\psi,\epsilon,R} \subseteq \lattice$ of measure at least $\frac{1}{1000}$ so that for each $\Lambda \in \Gc_{\psi,\epsilon,R}$, there exist:
    \begin{enumerate}[label=(\arabic*)]
        \item\label{perturbations} Perturbations $\eta_\Lambda \colon [0,1]^{n-1} \to \R$ so that $\norm{\eta_\Lambda}_{C^k} \lesssim \epsilon$, and with associated submanifolds
        \begin{equation*}
            \Sigma_\Lambda = \left\{ (\xi,(\psi + \eta_\Lambda)(\xi)) : \xi \in [0,1]^{n-1} \right\};
        \end{equation*}
        \item\label{weights'} weights $w_\Lambda \colon \B_R^n \to [0,\infty)$ satisfying:
        \begin{enumerate}[label=(\roman*)]
            \item\label{propertyi} $\norm{w_\Lambda}_{L^1} \lesssim R^{n-1}$;
            \item\label{propertyii} $\norm{Xw_\Lambda}_{L^\infty} \lesssim R^{\frac{n-1}{N}}$;
            \item\label{propertyiii} $\norm{\hat{w}_\Lambda}_{L^2(\Sigma_{f})}^2 \gtrsim R^qR^{n-1}$ for any $f \in C^k$ with $\norm{f - (\psi+\eta_\Lambda)}_{C^k} \ll R^{-1}$.
        \end{enumerate}
    \end{enumerate}
\end{prop}

\begin{rmk}
    We note that, in view of Remark \ref{rmk:highprob}, we may in fact take $\Gc_{\psi,\epsilon,R}$ to have measure $1-o(1)$.
\end{rmk}

\subsection{Construction of perturbations}

To start, we apply Theorem \ref{thm:countingnearmanifolds} in the case of a hypersurface to establish that a $C^k$-small perturbation of the surface $\Sigma$ can hit many disparate lattice points. This will be used to establish Proposition \ref{prop:mainproposition}\ref{weights'}\ref{propertyiii}.

\begin{lemma}\label{lemma:mainprobabilisticlemma}
    There exists a subset $\Gc_{\psi,\epsilon,R}' \subseteq \lattice$ of measure at least $\frac{1}{100}$ for which there exists a perturbation $\eta_\Lambda$ of $\psi$ which is $\epsilon$-close in the $C^k$-topology, and so that the intersection
    \begin{equation*}
        \pi(R^{-1/N}\Lambda \cap \B_{\rho_0}^N) \cap \Sigma_\Lambda
    \end{equation*}
    contains $\gtrsim \epsilon R^q$ points which are $\gtrsim R^{-\beta}$-separated.\footnote{Here, we are permitted to absorb $\rho_0$ into the implicit constants, since it does not depend on $\epsilon,R$.}
\end{lemma}
\begin{proof}
    We apply Theorem \ref{thm:countingpointsnge3} with $m = n-1$. For each $\theta \in \Theta$, choose (if possible) a point
    \begin{equation*}
        \vbf \in \pi(R^{-1/N}\Lambda \cap \B_1^N) \cap \Sigma_\psi(\epsilon,\theta).
    \end{equation*}
    Note that, with probability at least $\frac{1}{100}$, we are allowed to choose $\gtrsim \epsilon R^q$ such points $\vbf$. Let us additionally define the projections
    \begin{equation*}
        \pi_{\hor} \colon \R^{n-1} \times \R \to \R^{n-1}, \quad \quad \pi_{\ver} \colon \R^{n-1} \times \R \to \R.
    \end{equation*}
    For each $\vbf$, let $\chi_\vbf \colon \R^{n-1} \to [0,\infty)$ be a smooth bump function satisfying
    \begin{equation*}
        \ind_{\B_{R^{-\beta}/4}(\pi_{\hor}(\vbf))} \le \chi_\vbf \lesssim \ind_{\B_{R^{-\beta}/2}(\pi_{\hor}(\vbf))}.
    \end{equation*}
    Note that the functions $\chi_\vbf$ have disjoint supports by the separation condition on $\Theta$, and moreover we can choose these bump functions to satisfy
    \begin{equation*}
        \norm{\chi_\vbf}_{C^k} \lesssim R^{k\beta}.
    \end{equation*}
    We define perturbations $\eta_\Lambda$ by
    \begin{equation*}
        \eta_\Lambda = \sum_\vbf a_\vbf \chi_\vbf,
    \end{equation*}
    where $a_\vbf = \pi_{\ver}(\vbf) -\psi(\pi_{\hor}(\vbf))$. By construction, we have $|a_\vbf| \lesssim \epsilon R^{-k\beta}$, and so
    \begin{equation*}
        \norm{\eta_\Lambda}_{C^k} \lesssim \epsilon.
    \end{equation*}
    Therefore, $\Sigma_\Lambda = \{(\xi,(\psi+\eta_\Lambda)(\xi)) : \xi \in [0,1]^{n-1}\}$ satisfies the desired conditions.
\end{proof}

\subsection{Construction of weights}

The goal of this section is to prove Proposition \ref{prop:mainproposition}\ref{weights'}, which will conclude the proof of Theorem \ref{thm:LocalMTFailsalways}. From now on, we treat $\Lambda \in \lattice$ as a $\mu$-distributed random lattice. Let us define the functions
\begin{equation*}
    H_R = R^{-1}\phi_R * (b_1\delta_{R^{-1/N}\Lambda}), \quad \hat{h}_R = \hat{H}_R \circ \iota, \quad W_R = \hat{H}_R^2, \quad w_R = \hat{h}_R^2.
\end{equation*}
We may alternately express
\begin{equation*}
    w_R(x) = \left| \left( \phi_1 * \sum_{\xi \in R^{1/N}\Lambda^\vee} \delta_\xi \right)(\iota x)b_R(\iota x) \right|^2.
\end{equation*}

We will need two lemmas. The first one will help to bound the main terms in all of our estimates.

\begin{lemma}\label{lemma:diagonalestimate}
    If $K$ is an $n$-dimensional centrally symmetric convex body in $\R^N$, then
    \begin{equation*}
        \sum_{\xi \in R^{1/N}\Lambda^\vee} \int_{K+y} \phi_1^2(x-\xi)\dx \lesssim |R^{1/N}\Lambda^\vee \cap \Nc_2(2K)|
    \end{equation*}
    uniformly in $y \in \R^N$ and $\Lambda \in \lattice$.
\end{lemma}
\begin{proof}
    We decompose $\phi_1^2(x) = \int_0^A \ind_{S_t}(x)\dt$, where $S_t = \{y : \phi_1^2(y) \ge t\}$. Then
    \begin{align*}
        \sum_{\xi \in R^{1/N}\Lambda^\vee}\int_{K+y} \phi_1^2(x-\xi)\dx &= \int_0^A \sum_{\xi \in R^{1/N}\Lambda^\vee} \int_{K+y} \ind_{S_t}(x-\xi)\dx\dt \\
        &\le \int_0^A \sum_{\xi \in R^{1/N}\Lambda^\vee} \int_{K+y}\ind_{\xi + \B_{C_Mt^{-1/M}}}(x)\dx\dt \\
        &\lesssim_M \int_0^A \sum_{\xi \in R^{1/N}\Lambda^\vee} t^{-n/M} \ind_{\xi \in \Nc_{C_Mt^{-1/M}}(K+y)}\dt \\
        &= \int_0^A t^{-n/M}\left|R^{1/N}\Lambda^\vee \cap \Nc_{C_Mt^{-1/M}}(K+y)\right|\dt.
    \end{align*}
    Now, suppose $\xi_1,...,\xi_s \in \Nc_{a}(K+y) \cap L$ for some lattice $L$. Then $\zeta_1,...,\zeta_s \in \Nc_{2a}(2K)$, where $\zeta_i = \xi_i - \xi_1$. Since we can cover $\Nc_{C_Mt^{-1/M}}(K+y)$ by $\lesssim_M t^{-N/M}$ translates of $\Nc_1(K)$, we obtain
    \begin{equation*}
        \left|R^{1/N}\Lambda^\vee \cap \Nc_{C_Mt^{-1/M}}(K+y)\right| \lesssim_M t^{-N/M}\left|R^{1/N}\Lambda^\vee \cap \Nc_{2}(2K)\right|.
    \end{equation*}
    Therefore, we obtain
    \begin{equation*}
        \int_0^A t^{-n/M}\left|R^{1/N}\Lambda^\vee \cap \Nc_{C_Mt^{-1/M}}(K+y)\right|\dt \lesssim_M \left|R^{1/N}\Lambda^\vee \cap \Nc_{2}(2K)\right|\int_0^A t^{-(N+n)/M}\dt,
    \end{equation*}
    and so taking $M > N+n$ yields the desired inequality.
\end{proof}

The second lemma will bound the off-diagonal terms.

\begin{lemma}\label{lemma:offdiagonalestimate}
    With probability at least $\frac{997}{1000}$, we have
    \begin{equation*}
        \sum_{\substack{\zeta \in R^{1/N}\Lambda^\vee \\ \zeta \ne 0}}  \ \sum_{\xi \in R^{1/N}\Lambda^\vee} \phi_1(x-\xi)\phi_1(x-\xi-\zeta) \lesssim R^{-1},
    \end{equation*}
    independent of $x \in \R^N$.
\end{lemma}
\begin{proof}
    To start, we fix $\zeta \in R^{1/N}\Lambda^\vee \setminus \{0\}$ and define
    \begin{equation*}
        F_\zeta(x) = \sum_{\xi \in R^{1/N}\Lambda^\vee} \phi_1(x-\xi)\phi_1(x-\xi-\zeta).
    \end{equation*}
    We split this sum into two regimes, depending on whether $|x-\xi|$ is less than or greater than $|\zeta|/2$. Fix some $M \gg N$. The first contribution can be bounded by
    \begin{align*}
        \sum_{\substack{\xi \in R^{1/N}\Lambda^\vee \\ |x-\xi| \le |\zeta|/2}}(1+|x-\xi|)^{-M}(1+|x-\xi-\zeta|)^{-M} &\lesssim \left| R^{1/N}\Lambda^\vee \cap \B_{|\zeta|/2}(x) \right|(1+|\zeta|)^{-M} \\
        &\lesssim |\Lambda^\vee \cap \B_{R^{-1/N}|\zeta|/2}(R^{-1/N}x)|(1+|\zeta|)^{-M}.
    \end{align*}
    To bound the second contribution, we note
    \begin{align*}
        \sum_{\substack{\xi \in R^{1/N}\Lambda^\vee \\ |x-\xi| > |\zeta|/2}}(1+|x-\xi|)^{-M}\phi_1(x-\xi-\zeta)&\lesssim (1+|\zeta|)^{-M}\sum_{\xi \in R^{1/N}\Lambda^\vee}\phi_1(x-\xi-\zeta).
    \end{align*}
    Applying Poisson summation, we get
    \begin{align*}
        \sum_{\xi \in R^{1/N}\Lambda^\vee}\phi_1(x-\xi-\zeta) &\le R^{-1}\sum_{\vbf \in R^{-1/N}\Lambda} b_1(\vbf) \\
        &\lesssim R^{-1}|R^{-1/N}\Lambda \cap \B_1^N| \\
        &= R^{-1}|\Lambda \cap \B_{R^{1/N}}|.
    \end{align*}
    Thus, it suffices to bound
    \begin{equation*}
        (1+|\zeta|)^{-M}\left(|\Lambda^\vee \cap \B_{R^{-1/N}|\zeta|/2}(R^{-1/N}x)| + R^{-1}|\Lambda \cap \B_{R^{1/N}}| \right).
    \end{equation*}
    Let us pass to a $\frac{999}{1000}$-measure subset for which $\alpha_1(\Lambda) \gtrsim 1$. We may then apply Lemma \ref{lemma:sysbounds} to obtain
    \begin{equation*}
        |\Lambda^\vee \cap \B_{R^{-1/N}|\zeta|/2}(R^{-1/N}x)| \lesssim |\zeta|^N,
    \end{equation*}
    with the inequality holding uniformly in $R,\zeta$. Moreover, it is easy to check
    \begin{equation*}
        \E[|\Lambda \cap \B_{R^{1/N}}|] \lesssim R.
    \end{equation*}
    Therefore, by Mahler's compactness criterion and Markov's inequality we may pass to a $\frac{998}{1000}$-measure subset of lattices on which we have
    \begin{equation*}
        |\Lambda^\vee \cap \B_{R^{-1/N}|\zeta|/2}(R^{-1/N}x)| + R^{-1}|\Lambda \cap \B_{R^{1/N}}| \lesssim |\zeta|^N + 1,
    \end{equation*}
    so summing in $\zeta$ yields
    \begin{align*}
        \sum_{\substack{\zeta \in R^{1/N}\Lambda^\vee \\ \zeta \ne 0}} F_\zeta(x) &\lesssim \sum_{\substack{\zeta \in R^{1/N}\Lambda^\vee \\ \zeta \ne 0}}(1+|\zeta|)^{-M}(|\zeta|^N+1) \\
        &\lesssim\sum_{\substack{\zeta \in R^{1/N}\Lambda^\vee \\ \zeta \ne 0}} (1+|\zeta|)^{-M+N} \\
        &= \sum_{\xi \in \Lambda^\vee \setminus \{0\}} (1+R^{1/N}|\xi|)^{-M+N}.
    \end{align*}
    Note that this final sum is the (non-primitive) Siegel transform of $f(y) = (1+R^{1/N}|y|)^{-M+N}$, so
    \begin{equation*}
        \E\left(\, \sum_{\xi \in \Lambda^\vee \setminus \{0\}} (1+R^{1/N}|\xi|)^{-M+N} \right) = \int_{\R^N} (1+R^{1/N}|y|)^{-M+N} \dy \lesssim R^{-1},
    \end{equation*}
    and applying Markov's inequality we obtain the desired conclusion after restricting to a $\frac{997}{1000}$-measure subset.
\end{proof}

\subsubsection{First property}

In this section, we establish \ref{prop:mainproposition}\ref{weights'}\ref{propertyi}.

\begin{lemma}\label{lemma:propertyi}
    For all $R \ge 1$, we have
\begin{equation*}
    \P\left(\norm{w_R}_{L^1} \lesssim R^{n-1}\right) \ge \frac{996}{1000}. 
\end{equation*}
\end{lemma}
\begin{proof}
    We have
    \begin{align*}
        \norm{w_R}_{L^1} &\lesssim \int_{\iota\B_{2R}} \sum_{\xi,\zeta \in R^{1/N}\Lambda^\vee} \phi_1(x-\xi)\phi_1(x-\zeta)\dx \\
        &= \sum_{\xi \in R^{1/N}\Lambda^\vee} \int_{\iota \B_{2R}} \phi_1^2(x-\xi)\dx + \sum_{\substack{\zeta \in R^{1/N}\Lambda^\vee \\ \zeta \ne 0}} \ \sum_{\xi \in R^{1/N}\Lambda^\vee}\int_{\iota\B_{2R}}\phi_1(x-\xi)\phi_1(x-\xi - \zeta)\dx.
    \end{align*}
    To control the main term, we use Lemma \ref{lemma:diagonalestimate} to obtain\footnote{Note that we could alternatively control the main term using the John ellipse and an analogous argument to the proof of Theorem \ref{thm:cuspestimatefinitetimeprobupper}.}
    \begin{align*}
        \sum_{\xi \in R^{1/N}\Lambda^\vee} \int_{\iota \B_{2R}}\phi_1^2(x-\xi)\dx &\lesssim |R^{1/N}\Lambda^\vee \cap \Nc_2(\iota \B_{4R})| = |\Lambda^\vee \cap R^{-1/N}\Nc_2(\iota \B_{4R})|.
    \end{align*}
    Since $\Vol(R^{-1/N}\Nc_2(\iota \B_{4R})) \asymp R^{n-1}$, we may apply Markov's inequality to obtain that this term is $\lesssim R^{n-1}$ with probability at least $\frac{999}{1000}$. The off-diagonal term is $\lesssim R^{n-1}$ with probability at least $\frac{997}{1000}$ by Lemma \ref{lemma:offdiagonalestimate}.
\end{proof}

\subsubsection{Second property}

In this section we establish Proposition \ref{prop:mainproposition}\ref{weights'}\ref{propertyii} with error $\lesssim R^{\frac{n-1}{N}}$.

\begin{lemma}\label{lemma:propertyii}
    For all $R \gg 1$, we have
    \begin{equation*}
        \P\left( \norm{Xw_R}_{L^\infty} \lesssim R^{\frac{n-1}{N}} \right) \ge \frac{999}{1000}.
    \end{equation*}
\end{lemma}
\begin{proof}
    Fix a line $\ell \subseteq V$, and note
    \begin{align*}
        \int_\ell w_R  &\lesssim \int_{\ell \cap \B_{2R}} \sum_{\xi \in R^{1/N}\Lambda^\vee} \phi_1^2(x-\xi)\dx  \\
        &\quad \quad + \sum_{\substack{\zeta \in R^{1/N}\Lambda^\vee \\ \zeta \ne 0}} \ \sum_{\xi \in R^{1/N}\Lambda^\vee} \int_{\ell \cap \B_{2R}} \phi_1(x-\xi)\phi_1(x-\xi-\zeta)\dx,
    \end{align*}
    and the error term is $O(1)$ in view of Lemma \ref{lemma:offdiagonalestimate}.
    
    Let us now bound the main term. Again, we may apply Lemma \ref{lemma:diagonalestimate} to obtain
    \begin{equation*}
        \int_{\ell \cap \B_{2R}} \sum_{\xi \in R^{1/N}\Lambda^\vee} \phi_1^2(x-\xi)\dx \lesssim |R^{1/N}\Lambda^\vee \cap \Nc_{2}(\ell_0 \cap \B_{4R})| = |\Lambda^\vee \cap R^{-1/N}\Nc_2(\ell_0 \cap \B_{4R})|,
    \end{equation*}
    where $\ell_0 \in \Gr_1(V)$ is the line parallel to $\ell$.\footnote{We write $\Gr_1(V)$ for the projective space (or 1-Grassmannian) of $V$ as opposed to the more conventional $\P(V)$, so as to avoid conflicting notation.}  Therefore, up to a uniform error, it suffices to consider 
    \begin{equation*}
        \P \left( \sup_{\ell \in \Gr_1(V)} |\Lambda \cap T_\ell| \ge A \right),
    \end{equation*}
    where $T_\ell$ is the tube around $\ell$ of length $\asymp R^{1-1/N}$ and radius $\asymp R^{-1/N}$.

    Now, take $\zbf = (1-1/N,-1/N,\dots,-1/N) \in \afrak_+$, and set $R = e^t$. Note $T_\ell \approx hg^{-t}\B_1^N$ for some $h \in H = \SO(n) \times \Id_{N-n}$, where $S \approx S'$ means that $C^{-1}S \subseteq S' \subseteq CS$ for some $C = O(1)$. By Corollary \ref{cor:height=count}, we find
    \begin{equation*}
        \P \left( \sup_{\ell \in \Gr_1(V)} |\Lambda \cap T_\ell| \ge A \right) \le \P\left( \Max_\zbf^t(\Lambda) \gtrsim A \right).
    \end{equation*}
    Applying Theorem \ref{thm:cuspestimatefinitetimeprobupper}, we see that
    \begin{equation*}
        \P\left( \Max_\zbf^t(\Lambda) \ge Te^{\delta_\zbf t/N} \right) \lesssim T^{-N},
    \end{equation*}
    and so we have for all $t \ge 1$ that
    \begin{equation*}
        \P\left( \Max_\zbf^t(\Lambda) \lesssim e^{\delta_\zbf t/N} \right) \ge \frac{999}{1000}.
    \end{equation*}
    But now, we compute
    \begin{equation*}
        \delta_\zbf = \sum_{1 \le i < j \le n}(z_i - z_j) = (n-1)z_1 - \sum_{j=2}^n z_j = (n-1)(N-1)/N + (n-1)/N = n-1,
    \end{equation*}
    and so $e^{\delta_\zbf t/N} = R^{(n-1)/N}$, which finishes the proof.
\end{proof}

\begin{rmk}
    The bound $\lesssim R^{(n-1)/N}$ is optimal for our choice of weights; see Remark \ref{rmk:matchingbounds}.
\end{rmk}

\subsubsection{Third property}

\begin{lemma}\label{lemma:propertyiii}
    For all $R \gg 1$, we have
    \begin{equation*}
        \P\left( \norm{\hat{w}_R}_{L^2(\Sigma_\Lambda)}^2 \gtrsim \epsilon R^\frac{n-1}{n-1+k}R^{n-1} \right) \ge \frac{9}{1000}.
    \end{equation*}
\end{lemma}
\begin{proof}
    We follow the strategy of \cite{CaiZha}, this time without any power loss. By the projection-slice theorem, we see that\footnote{Here, as in \cite{CaiZha}, we abuse notation by writing $\phi_R$ for the appropriately scaled functions in both $n$ and $N$ dimensions.}
    \begin{equation*}
        h_R = R^{-1}\phi_R * \left( b_1\delta_{\pi(R^{-1/N} \Lambda)} \right),
    \end{equation*}
    and so
    \begin{equation*}
        \hat{w}_R = h_R * h_R = R^{-2}\phi_R * \phi_R * (b_1\delta_{\pi(R^{-1/N} \Lambda)}) * (b_1\delta_{\pi(R^{-1/N} \Lambda)}).
    \end{equation*}
    
    Let us introduce the temporary notation
    \begin{equation*}
        B_c(\Lambda) = R^{-1/N}\Lambda \cap \B_c^N.
    \end{equation*}
    By restricting to a $\frac{999}{1000}$-measure compact subset of $\lattice$ so that the geometry of $\Lambda$ is uniformly controlled, we can find some $0 < c < 1/2$ so that for each $\ubf \in B_c(\Lambda)$, one has $\gtrsim |B_c(\Lambda)|$ solutions $\vbf,\wbf \in B_{2c}(\Lambda)$ with $\vbf-\wbf = \ubf$. Using that $\phi_R * \phi_R \gtrsim R^n\ind_{\B_{R^{-1}}}$, we thus obtain
    \begin{equation*}
        \hat{w}_R \gtrsim R^{n-2}|B_c(\Lambda)|\ind_{\Nc_{R^{-1}}(\pi(B_c(\Lambda)))}.
    \end{equation*}
    Now, recall that $\Sigma_\Lambda$ is a random perturbation with good intersection properties with $\Lambda$, as in Lemma \ref{lemma:mainprobabilisticlemma}. Decompose
    \begin{align*}
        \norm{\hat{w}_R}_{L^2(\Sigma_\Lambda)}^2 &\gtrsim \int_{\Sigma_\Lambda \cap \Nc_{R^{-1}}(\pi B_c(\Lambda))}R^{2n-4}|B_c(\Lambda)|^2\dsigma \ge R^{2n-4}|B_c(\Lambda)|^2\sigma\left( \Sigma_\Lambda \cap \Nc_{R^{-1}}(\pi B_c(\Lambda))\right).
    \end{align*}
    Using Lemma \ref{lemma:mainprobabilisticlemma} for $\rho_0 = c$, we see that $|\Sigma_\Lambda \cap \pi B_{\rho_0}(\Lambda)| \gtrsim \epsilon R^q$ with probability at least $\frac{1}{100}$, and moreover that the points in this intersection are $R^{-1}$-separated. Thus, we have
    \begin{equation*}
        \sigma\left( \Sigma_\Lambda \cap \Nc_{R^{-1}}(\pi B_c(\Lambda))\right) \gtrsim \epsilon R^{\frac{n-1}{n-1+k}}R^{-n+1},
    \end{equation*}
    and so in total we have
    \begin{equation*}
        \norm{\hat{w}_R}_{L^2(\Sigma_\Lambda)}^2 \gtrsim \epsilon R^{\frac{n-1}{n-1+k}}R^{n-1}\left(R^{-2}|B_{\rho_0}(\Lambda)|^2\right).
    \end{equation*}
    Applying Theorem \ref{thm:VDC}, we see that $|B_{\rho_0}(\Lambda)| \gtrsim R$, which completes the bound.

    Finally, we note that whenever $\norm{f - (\psi + \eta_\Lambda)}_{C^k} \ll R^{-1}$, the same lower bound holds for $\hat{w}_R$ integrated over $\Sigma_f$, since we integrate over an $R^{-1}$-neighborhood of $\Sigma_\Lambda$.
\end{proof}

\subsection{Proof of Theorem \ref{thm:LocalMTFailsalways}}

Finally, we are able to complete the proof of power loss for local Mizohata-Takeuchi.

\begin{proof}[Proof of Theorem \ref{thm:LocalMTFailsalways}]
    By combining Lemmas \ref{lemma:propertyi}, \ref{lemma:propertyii}, \ref{lemma:propertyiii}, we see that we can find for each $R \ge C_0 \rho_0^{-(N-n)/q}\epsilon^{-1/q}$ a set of lattices in $\Gc_{\psi,\epsilon,R} \subseteq \Gc_{\psi,\epsilon,R}'$ of measure
    \begin{equation*}
        \mu(\Gc_{\psi,\epsilon,R}) \ge \frac{9}{1000} - \frac{4}{1000} - \frac{1}{1000} \ge \frac{1}{1000}
    \end{equation*}
    for which all three portions of Proposition \ref{prop:mainproposition}\ref{weights'} hold, which completes the proof. In view of Lemma \ref{lemma:errorreduction}, this implies Theorem \ref{thm:LocalMTFailsalways}.
\end{proof}

%% file: Sections/7-genericity.tex
\section{Genericity of power loss}\label{section:genericity}

In this section, we prove Theorem \ref{thm:localMTgenericallyfails}. To this end, fix some $\delta > 0$, and choose $N \ge 2n\delta^{-1}$. Let us also write $\Psi \subseteq C^k([0,1]^{n-1})$ for the unit ball in the $C^k$-norm.

\begin{proof}[Proof of Theorem \ref{thm:localMTgenericallyfails}]
    Suppose that $0 < \epsilon < 1$, and $R \gg \epsilon^{-1/q}$. By Proposition \ref{prop:mainproposition}, there exists some open subset $V_{\epsilon,R} \subseteq \Psi$ which contains the $\gtrsim R^{-1}$-neighborhoods of a $\lesssim \epsilon$-dense subset in $\Psi$, so that for every $\psi \in V_{\epsilon,R}$ there exists a weight $w \colon \B_R \to [0,\infty)$ for which
    \begin{equation*}
        \norm{\hat{w}}_{L^2(\Sigma_\psi)}^2 \gtrsim \epsilon R^qR^{-\frac{n-1}{N}}\norm{w}_{L^1}\norm{Xw}_{L^\infty}.
    \end{equation*}

    Now, let $(\epsilon_k)_{k \in \N}$ be a sequence tending to zero, and suppose $R_k \gg \epsilon_k^{-A}$ for some $A > 2\delta^{-1}$. Define the set
    \begin{equation*}
        V_K = \bigcup_{k \ge K} V_{\epsilon_k,R_k}.
    \end{equation*}
    Then $V_K$ is clearly open and dense. By the Baire category theorem, the set
    \begin{equation*}
        V = \bigcap_{K \ge 0}V_K = \limsup_{k \to \infty} V_{\epsilon_k,R_k}
    \end{equation*}
    is a dense $G_\delta$. Moreover, if $\psi \in V$, then there exist arbitrarily large $R_k \ge 1$ for which we can find a weight $w$ on $\B_{R_k}$ so that
    \begin{equation*}
        \norm{\hat{w}}_{L^2(\Sigma_\psi)}^2 \gtrsim R_k^qR_k^{-1/A}R_k^{-\frac{n-1}{N}}\norm{w}_{L^1}\norm{Xw}_{L^\infty} \ge R_k^q R_k^{-\delta}\norm{w}_{L^1}\norm{Xw}_{L^\infty}.
    \end{equation*}
\end{proof}

\begin{rmk}
    Just like in Theorem \ref{thm:LocalMTFailsalways}, one can obtain a quantitative bound on the error rate $\lesssim_\delta R^\delta$; we have made no effort to do so here.
\end{rmk}

%% file: Sections/8-bookkeeping.tex
\section{Bookkeeping and explicit error bounds}\label{section:bookkeeping}

In this section, we consider the explicit nature of all dependencies on $N$ to obtain the error term stated in Theorem \ref{thm:LocalMTFailsalways}. We frequently have dependencies which are (at worst) factorial in $N$, meaning $A \lesssim N^{KN} \cdot B$ for some $K > 0$ and with the implicit constant now independent of $N$. In particular, this holds for all bounds used to prove Lemmas \ref{lemma:propertyi} and \ref{lemma:propertyiii}.

Establishing Lemma \ref{lemma:propertyii}, on the other hand, requires a bit more care (the asymptotics are still of shape $N^{O(N)}$). We state the following quantitative analogues of our qualitative results.

\begin{thm}[Quantitative Analogue of Theorem \ref{thm:cuspvolume}]\label{thm:cuspvolumequant}
    For all $N \ge 2$ and all $0 < \epsilon < 1$, we have
    \begin{equation*}
        \mu\left( \lattice^{\thin,\epsilon} \right) \asymp \frac{\pi^{N/2}}{N \cdot \Gamma(N/2) \zeta(N)}\epsilon^N,
    \end{equation*}
    with the implicit constant independent of $N,\epsilon$.
\end{thm}

The proof of Theorem \ref{thm:cuspvolumequant} is established in Appendix \ref{section:cuspvolumes}; note that the more standard argument involving Siegel sets given in the proof of Theorem \ref{thm:cuspvolume} yields a bound of shape $\lesssim C^{N^3}\epsilon^N$ for some absolute $C > 1$ by Remark \ref{rmk:suboptimalcuspbound}; this would lead to an error bound of shape $\exp(O(\log^{2/3}R))$ in Theorem \ref{thm:LocalMTFailsalways}.

\begin{prop}[Quantitative Analogue of Corollary \ref{cor:height=count}]\label{prop:quantheightcount}
    For all $\Lambda \in \lattice$, we have
    \begin{equation*}
        \frac{1}{6^NN!}|\B_1^N \cap \Lambda| \le \alpha(\Lambda) \le 2^N|\B_1^N \cap \Lambda|.
    \end{equation*}
\end{prop}
\begin{proof}
    To start, let $\Delta \subseteq \Lambda$ be the lattice spanned by $\Lambda \cap \B_1^N$, and let $d = \dim \Delta$. The quantitative version of Theorem \ref{thm:succminest} implies
    \begin{equation*}
        |\B_1^d \cap \Delta| \le 2^{d-1}\prod_{i=1}^d \left(\frac{2}{\lambda_i(\Delta)} + 1\right) \le 6^d \left(\,\prod_{i=1}^d \lambda_i(\Delta) \right)^{-1}.
    \end{equation*}
    Then Minkowski's second theorem implies
    \begin{equation*}
        \left(\,\prod_{i=1}^d \lambda_i(\Delta) \right)^{-1} \le \frac{d!}{2^d}\frac{\Vol(\B_1^d)}{\Covol(\Delta)} \le d!\cdot \alpha(\Delta).
    \end{equation*}
    Thus, in total, we have
    \begin{equation*}
        \alpha(\Lambda) \ge \frac{1}{6^NN!}|\B_1^d \cap \Delta|.
    \end{equation*}

    In the other direction, let us assume the supremum in $\alpha(\Lambda)$ is realized by a $d$-dimensional sublattice $\Delta \subseteq \Lambda$. By Theorem \ref{thm:VDC}, we have
    \begin{equation*}
        |\B_1^d \cap \Delta| \ge \frac{\Vol(\B_1^d)}{2^d\Covol(\Delta)} \ge 2^{-d}\alpha(\Lambda).
    \end{equation*}
    Thus we obtain
    \begin{equation*}
        \alpha(\Lambda) \le 2^N|\B_1^N \cap \Lambda|.
    \end{equation*}
\end{proof}

We also need a quantitative version of Lemma \ref{lemma:propertyii}.

\begin{lemma}\label{lemma:quantativeprobbound}
    There exists some $C_0 > 1$, independent of $N$, so that for all $R \gg 1$ we have
    \begin{equation*}
        \P\left( \norm{Xw}_{L^\infty} \le C_0600^NN^{2N}R^{\frac{n-1}{N}} \right) \ge \frac{999}{1000}.
    \end{equation*}
\end{lemma}
\begin{proof}
    If we replace tubes $T_\ell$ in the proof of Lemma \ref{lemma:propertyii} with cubes of the same dimensions, we obtain
\begin{align*}
    \P \left( \sup_{\ell \in \Gr_1(V)} |\Lambda \cap T_\ell| \ge A \right) &\le \P\left(\, \sup_{h \in H} |g^{t}h\Lambda \cap \B_{2\sqrt{N}}^N| \ge A \right).
\end{align*}
By subdividing this ball into $\le (100N)^N$ translates of $\B_{\frac{1}{4\sqrt{N}}}^N$ and using an argument as in the proof of Lemma \ref{lemma:diagonalestimate}, we see that
\begin{align*}
    \P\left(\, \sup_{h \in H} |g^{t}h\Lambda \cap \B_{2\sqrt{N}}^N| \ge A \right) &\le \P\left(\, \sup_{h \in H} |g^{t}h\Lambda \cap \B_{\frac{1}{2\sqrt{N}}}^N| \ge (100N)^{-N}A \right).
\end{align*}
Using the same net argument as in Lemma \ref{lemma:netsize}, we find that there exists some constants $c_1,c_2 > 0$ (independent of $N$) for which
\begin{equation*}
    \P\left(\, \sup_{h \in H} |g^{t}h\Lambda \cap \B_{\frac{1}{2\sqrt{N}}}^N| \ge (100N)^{-N}A \right) \le c_1e^{\delta_\zbf t}\P(|\Lambda \cap \B_1^N| \ge c_2(100N)^{-N}A).
\end{equation*}
Now, combining Theorem \ref{thm:cuspvolumequant} and Proposition \ref{prop:quantheightcount}, we obtain
\begin{align*}
    \P(|\Lambda \cap \B_1^N| \ge Te^{\delta_\zbf t/N}) &\le \P\left( \alpha(\Lambda) \ge (6N)^{-N}Te^{\delta_\zbf t/N} \right) \\
    &\lesssim (6N)^{N^2}\frac{\pi^{N/2}}{N\cdot \Gamma(N/2)\zeta(N)}T^{-N}e^{-\delta_\zbf t}.
\end{align*}
Therefore, when $N \gg 1$ we obtain
\begin{align*}
    \P \left(\, \sup_{\ell \in \Gr_1(V)} |\Lambda \cap T_\ell| \ge TR^{\frac{n-1}{N}} \right) &\le c_1 R^{n-1}\P\left(|\Lambda \cap \B_1^N| \ge c_2(100N)^{-N}TR^{\frac{n-1}{N}}\right) \\
    &\le c_1c_2^{-N}600^{N^2}N^{2N^2}T^{-N}.
\end{align*}
Now, we require
\begin{equation*}
    c_1c_2^{-N}600^{N^2}N^{2N^2}T^{-N} \le \frac{1}{1000},
\end{equation*}
so it suffices to choose
\begin{equation*}
    T \ge (1000c_1c_2^{-1})600^N N^{2N} = N^{O(N)}.
\end{equation*}
\end{proof}

Combining Lemma \ref{lemma:quantativeprobbound} with the $N^{O(N)}$-error terms coming from properties (i) and (ii), we see that
\begin{align*}
    \norm{\hat{w}}_{L^2(\Sigma_R)}^2 &\ge N^{-O(N)}R^{\frac{n-1}{n-1+k}}R^{n-1} \\
    &\ge N^{-O(N)}N^{-O(N)}R^{\frac{n-1}{n-1+k}}\norm{w}_{L^1} \\
    &\ge N^{-O(N)}N^{-O(N)}N^{-O(N)} R^{\frac{n-1}{n-1+k}}R^{-\frac{n-1}{N}}\norm{Xw}_{L^\infty}\norm{w}_{L^1}.
\end{align*}
Therefore, there is $C_1 \ge 1$ so that for $N \gg 1$ and $R \gg 1$ we obtain
\begin{equation*}
    R^{\frac{n-1}{n-1+k}} \norm{\hat{w}}_{L^2(\Sigma_R)}^{-2} \norm{Xw}_{L^\infty} \norm{w}_{L^1} \le N^{C_1N}R^{\frac{n-1}{N}}.
\end{equation*}
For optimal estimates, we then take $N^2\log N \sim \log R$, after which we obtain
\begin{equation*}
    R^{\frac{n-1}{n-1+k}} \norm{Xw}_{L^\infty} \norm{w}_{L^1} \norm{\hat{w}}_{L^2(\Sigma_R)}^{-2} \le \exp\left(O(\sqrt{\log R \log\log R})\right).
\end{equation*}
This completes the proof of Theorem \ref{thm:LocalMTFailsalways}.

%% file: Sections/9-cuspvolumes.tex
\section{Cusp volumes}\label{section:cuspvolumes}

In this appendix, we establish Theorem \ref{thm:cuspvolumequant}, completing the proof of Theorem \ref{thm:LocalMTFailsalways}.

\subsection{Unfolding}

For each $1 \le p \le N-1$, write
\begin{equation*}
    \lattice^{\thin,\epsilon,p} = \left\{ \Lambda : \alpha_p(\Lambda) \ge \epsilon^{-1} \right\}
\end{equation*}
for the \textit{$p$-cusp} of $\lattice$. Let us also set $q = N-p$, so $p+q=N$. It will also be convenient to consider the (Archimedean portion of the) \textit{Tamagawa measure} $\omega$ on $G$, for which we have $\int_{G/\Gamma}\domega = \zeta(2) \cdots \zeta(N)$ (see \cite[Theorem 3.3.1]{Wei} or \cite{Gar}).

\begin{lemma}\label{lemma:pvolumesparabolic}
    For each $1 \le p \le N-1$, we have
    \begin{equation*}
        \int_{G/\Gamma} \ind_{\lattice^{\thin,\epsilon,p}} \domega \le \int_{G/(\Gamma \cap P)} \ind_{\Covol(g\Delta_0) \le \epsilon} \domega,
    \end{equation*}
    where $\Delta_0 = \Z^p \times 0^q$ and
    \begin{equation*}
        P = \left\{ \begin{bmatrix}
        X & Z \\
        0 & Y
    \end{bmatrix} : X \in \GL_p(\R),\ Y \in \GL_q(\R),\ \det(X)\det(Y) = 1,\ Z \in \Mat_{p \times q}(\R) \right\}.
    \end{equation*}
\end{lemma}
\begin{proof}
    The proof runs by unfolding:
    \begin{align*}
        \int_{\lattice}\ind_{\alpha_p(\Lambda) \ge \epsilon^{-1}} \domega
        &\le \int_{\lattice}\sum_{\substack{\Delta \le \Lambda \\ \text{$\Delta$ primitive} \\ \rk \Delta = p}} \ind_{\Covol(\Delta) \le \epsilon}\domega \\
        &= \int_{G/\Gamma} \sum_{\Delta \le \Z^N} \ind_{\Covol(g\Delta) \le \epsilon} \domega.
    \end{align*}
    $\Gamma$ acts transitively on the space of primitive rank-$p$ sublattices of $\Z^N$, and the stabilizer of $\Delta_0 = \Z^p \times 0^q$ is precisely $\Gamma \cap P$. Thus, we have
    \begin{equation*}
        \int_{G/\Gamma} \sum_{\Delta \le \Z^N} \ind_{\Covol(g\Delta) \le \epsilon} \domega = \int_{G/\Gamma} \sum_{\gamma \in \Gamma/(\Gamma \cap P)} \ind_{\Covol(g\gamma\Delta_0) \le \epsilon} \domega = \int_{G/(\Gamma \cap P)}\ind_{\Covol(g\Delta_0) \le \epsilon}\domega.
    \end{equation*}
\end{proof}

\subsection{Oriented double covers and fundamental domains}

In order to avoid unnecessary tracking of factors of 2, it will be convenient to work with the double covers:
\begin{equation*}
    \Tilde{G} = \SL_N^{\pm}(\R), \quad \Tilde{\Gamma} = \SL_N^{\pm}(\Z), \quad \Tilde{K} = \Ortho(N),
\end{equation*}
where $\SL_\ell^\pm(A) = \{g \in \GL_\ell(A) : \det(g) = \pm 1\}$. Given $p+q = N$, we consider the parabolic subgroup $\Tilde{P} \le \Tilde{G}$ defined by
\begin{equation*}
    \Tilde{P} = \left\{ \begin{bmatrix}
        X & Z \\
        0 & Y
    \end{bmatrix} : X \in \GL_p(\R),\ Y \in \GL_q(\R),\ \det(X)\det(Y) = \pm 1,\ Z \in \Mat_{p \times q}(\R) \right\}.
\end{equation*}
Let us additionally set $\Tilde{\Gamma}_P = \Tilde{\Gamma} \cap \Tilde{P}$. In view of Lemma \ref{lemma:pvolumesparabolic}, we would like to produce a natural fundamental domain for $\Tilde{\Gamma}_P \acts \Tilde{G}$.

\begin{lemma}\label{lemma:fundamentaldomains}
    There is a natural identification
    \begin{equation*}
        \Tilde{G}/\Tilde{\Gamma}_P \cong \Gr(p,N) \times \R_+ \times \Ls_p \times \Ls_q \times \T^{p \times q}.
    \end{equation*}
\end{lemma}
\begin{proof}
    To start, we note that $\Tilde{K}\Tilde{P} = \Tilde{G}$, and we have
    \begin{equation*}
        \Tilde{K} \cap \Tilde{P} = \Ortho(p) \times \Ortho(q).
    \end{equation*}
    Thus, it suffices to find a fundamental domain for $\Tilde{\Gamma}_P$ acting on $\Tilde{P}$, after which we may take a product with the compact quotient
    \begin{equation*}
        \Tilde{K}/(\Tilde{K} \cap \Tilde{P}) = \Gr(p,N).
    \end{equation*}
    Associated to $\Tilde{P}$ is the \textit{Langlands decomposition} $\Tilde{P} = MAU$, where
    \begin{align*}
        M &= \left\{ \begin{bmatrix}
            g_p & 0 \\
            0 & g_q
        \end{bmatrix} : g_p \in \SL_p^{\pm}(\R),\ g_q \in \SL_q^{\pm}(\R) \right\} \\
        A &= \left\{\begin{bmatrix}
            t^{1/p} \Id_p & 0 \\
            0 & t^{-1/q}\Id_q
        \end{bmatrix} : t \in \R_+\right\} \\
        U &= \left\{\begin{bmatrix}
            \Id_p & Z \\
            0 & \Id_q
        \end{bmatrix} : Z \in \Mat_{p \times q}(\R)\right\}.
    \end{align*}
    Then we have a decomposition
    \begin{equation*}
        \Tilde{\Gamma}_P = \Gamma_M \ltimes \Gamma_U,
    \end{equation*}
    where $\Gamma_M = \Tilde{\Gamma}_P \cap M$ and $\Gamma_U = \Tilde{\Gamma}_P \cap U$. Note that $\Tilde{\Gamma}_P \cap A$ is trivial. Thus, we may decompose
    \begin{equation*}
        \Tilde{P}/\Tilde{\Gamma}_P \cong A \times M/\Gamma_M \times U/\Gamma_U.
    \end{equation*}
    Note that $A \cong \R_+$. Next, we identify $U \cong \R^{p \times q}$, so that $\Gamma_U \cong \Z^{p \times q}$, and
    \begin{equation*}
        U/\Gamma_U \cong \R^{p \times q}/\Z^{p \times q} \cong \T^{p \times q}.
    \end{equation*}
    Finally, we identify $M \cong \SL_p^{\pm}(\R) \times \SL_q^{\pm}(\R)$, so that
    \begin{equation*}
        M/\Gamma_M \cong \left( \SL_p^{\pm}(\R) \times \SL_q^{\pm}(\R)\right)/\left(\SL_p^{\pm}(\Z) \times \SL_q^{\pm}(\Z) \right) \cong \Ls_p \times \Ls_q,
    \end{equation*}
    as desired.
\end{proof}

\subsection{The upper bound}

Let $\xi(s) = \pi^{-s/2}\Gamma(s/2)\zeta(s)$ be the completed zeta function, and define
\begin{equation*}
    c_p = \prod_{\ell=2}^p \xi(\ell) \quad \mathrm{and} \quad C_{p,q} = \frac{c_pc_q}{Nc_N},
\end{equation*}
with $c_p = 1$ if $p =1$.

\begin{lemma}\label{lemma:pcuspvolumeupper}
    For each $1 \le p \le N-1$, we have
    \begin{equation*}
        \mu\left( \lattice^{\thin,\epsilon,p} \right) \le C_{p,q}\epsilon^N.
    \end{equation*}
\end{lemma}
\begin{proof}
    By Lemma \ref{lemma:pvolumesparabolic}, we have
    \begin{equation*}
        \mu\left( \lattice^{\thin,\epsilon,p} \right) \le \frac{1}{\zeta(2)\cdots \zeta(N)}\int_{G/\Gamma_P} \ind_{\Covol(g\Delta_0) \le \epsilon}\domega.
    \end{equation*}
    We may identify $G/\Gamma_P \cong \Tilde{G}/\Tilde{\Gamma}_P$, and then apply Lemma \ref{lemma:fundamentaldomains} to obtain
    \begin{equation*}
        \int_{\Tilde{G}/\Tilde{\Gamma}_P} \ind_{\Covol(g\Delta_0) \le \epsilon}\domega = \int_{\Gr(p,N)}\int_{\R_+}\int_{\Ls_p}\int_{\Ls_q}\int_{\T^{p\times q}} \rho(a)\ind_{\Covol(kmau\Delta_0)\le \epsilon}\du\da\dm\dk,
    \end{equation*}
    where $\rho(a)$ is the Jacobian of $\mathrm{Ad}(a)|_U$. Under the identification $A \cong \R_+$ given by
    \begin{equation*}
        t \mapsto \begin{bmatrix}
            t^{1/p}\Id_p & 0 \\
            0 & t^{-1/q}
        \end{bmatrix},
    \end{equation*}
    we have $\da = \frac{\dt}{t}$, and $a = a^t \in A$ acts on $U$ via uniform scaling by $t^{\frac{1}{p} + \frac{1}{q}} = t^{\frac{N}{pq}}$; thus
    \begin{equation*}
        \rho(a^t) = (t^{\frac{N}{pq}})^{\dim U} = t^N.
    \end{equation*}
    Thus, we are reduced to computing the integral
    \begin{equation*}
        \int_{\Gr(p,N)} \int_{\R_+} \int_{\Ls_p} \int_{\Ls_q} \int_{\T^{p\times q}} t^{N-1}\ind_{\Covol(kma^tu\Delta_0)\le \epsilon}\du\dt\dm\dk.
    \end{equation*}
    Note that none of the $k,m,u$-factors affect the covolume, and $\Covol(a^t\Delta_0) = t$, so the above integral is simply given by
    \begin{equation*}
        \Vol(\Gr(p,N))\Vol(\Ls_p)\Vol(\Ls_q)\Vol(\T^{p \times q}) \int_0^\epsilon t^{N-1}\dt.
    \end{equation*}
    Thus, we obtain
    \begin{align*}
        \int_{G/\Gamma_P} \ind_{\Covol(g\Delta_0) \le \epsilon}\dmu(g) &= \frac{\prod_{\ell=1}^N\frac{2\pi^{\ell/2}}{\Gamma(\ell/2)}}{\left(\prod_{\ell=1}^p\frac{2\pi^{\ell/2}}{\Gamma(\ell/2)} \right)\left(\prod_{\ell=1}^q\frac{2\pi^{\ell/2}}{\Gamma(\ell/2)} \right)} \frac{\left(\prod_{\ell=2}^p\zeta(\ell)\right)\left(\prod_{\ell=2}^q\zeta(\ell) \right)}{\prod_{\ell=2}^N\zeta(\ell)}\frac{\epsilon^N}{N} \\
        &= \frac{c_pc_q}{Nc_N}\epsilon^N.
    \end{align*}
\end{proof}

\begin{cor}\label{cor:cuspvolumeupper}
    For all $N \ge 2$ and $0 < \epsilon < 1$, we have
    \begin{equation*}
        \mu\left( \lattice^{\thin,\epsilon} \right) \le (2+o_{N \to \infty}(1))\frac{\pi^{N/2}}{N \cdot \Gamma(N/2)\zeta(N)}\epsilon^N.
    \end{equation*}
\end{cor}
\begin{proof}
    This follows immediately from Lemma \ref{lemma:pcuspvolumeupper} and the rapid growth of $\xi(s)$ as $s \to \infty$ along the real axis.
\end{proof}

\subsection{The lower bound}

We compute the lower bound separately for the cases $N=2$ and $N \ge 3$.

\begin{lemma}\label{lemma:cuspvolumelowerbound2}
    Suppose $N=2$. Then
    \begin{equation*}
        \mu\left( \Ls_2^{\thin,\epsilon} \right) = C_{1,1}\epsilon^2.
    \end{equation*}
\end{lemma}
\begin{proof}
    Recall that we may identify $\SO(2)\backslash{\SL_2(\R)} = \H^2 = \{\tau = x+iy : y > 0\}$. Under this identification, a fundamental domain for the $\Gamma$-action is given by
    \begin{equation*}
        \Fc = \left\{ \tau : |\Re\tau| \le 1/2,\ |\tau| \ge 1 \right\}.
    \end{equation*}
    A point $\tau \in \Fc$ corresponds to the lattice
    \begin{equation*}
        \Lambda_\tau = (\Im \tau)^{-1/2}(\Z + \tau \Z).
    \end{equation*}
    Note that, so long as $\alpha(\Lambda_\tau) \ge 1$, we have
    \begin{equation*}
        \alpha(\Lambda_\tau) = \alpha_1(\Lambda_\tau) = (\Im \tau)^{1/2}.
    \end{equation*}
    Thus, in the natural area form $\frac{\dx\dy}{y^2}$ on $\H^2$, we have
    \begin{equation*}
        \int_{\H^2/\Gamma} \ind_{\alpha(\Lambda_\tau) \ge \epsilon^{-1}}\frac{\dx\dy}{y^2} = \int_{-1/2}^{1/2}\int_{y=\epsilon^{-2}}^\infty \frac{\dx\dy}{y^2} = \epsilon^2.
    \end{equation*}
    Now, we note that $\int_{\Fc} \frac{\dx\dy}{y^2} = \frac{\pi}{3}$, and so
    \begin{equation*}
        \mu(\Ls_2^{\thin,\epsilon}) = \frac{3}{\pi}\epsilon^2 = \frac{1}{2\xi(2)}\epsilon^2 = C_{1,1}\epsilon^2.
    \end{equation*}
\end{proof}

\begin{lemma}\label{lemma:cuspvolumelowerbound3+}
    Suppose $N \ge 3$. Then
    \begin{equation*}
        \mu\left( \lattice^{\thin,\epsilon,1} \right) \ge (1-o_{N \to \infty}(1))C_{1,N-1}\epsilon^N.
    \end{equation*}
\end{lemma}
\begin{proof}
    Note that $\alpha_1(\Lambda) \ge \epsilon^{-1}$ if and only if $\Lambda_{\prim} \cap \B_\epsilon^N \ne \emptyset$. Let us define the random variable $X = |\Lambda_{\prim} \cap \B_\epsilon^N|$; this is the primitive Siegel transform of $\ind_{\B_\epsilon}$. By the primitive version of Siegel's mean value theorem and Rogers' second moment formula, we have
    \begin{equation*}
        \E[X] = \frac{1}{\zeta(N)}\Vol(\B_\epsilon) \quad \mathrm{and} \quad \E[X^2] = \frac{1}{\zeta(N)^2}\Vol(\B_\epsilon)^2 + \frac{2}{\zeta(N)}\Vol(\B_\epsilon).
    \end{equation*}
    Thus, we have
    \begin{align*}
        \P(X > 0) \ge \frac{\E[X]^2}{\E[X^2]} &= \frac{\zeta(N)^{-2}\Vol(\B_\epsilon)^2}{\zeta(N)^{-2}\Vol(\B_\epsilon)^2 + 2\zeta(N)^{-1}\Vol(\B_\epsilon)} \\
        &= \left( 1 + \frac{2\zeta(N)}{\Vol(\B_\epsilon)} \right)^{-1} \\
        &= \left( 1 + \frac{N \cdot \Gamma(N/2)\zeta(N)}{\pi^{N/2}}\epsilon^{-N} \right)^{-1} \\
        &\ge (1-o_{N \to \infty}(1))\frac{\pi^{N/2}}{N \cdot \Gamma(N/2)\zeta(N)}\epsilon^N.
    \end{align*}
\end{proof}

\subsection{Proof of Theorem \ref{thm:cuspvolumequant}}

Combining the previous results, we obtain the main theorem of this section.

\begin{proof}[Proof of Theorem \ref{thm:cuspvolumequant}]
    By Corollary \ref{cor:cuspvolumeupper}, we have
    \begin{equation*}
        \mu\left( \lattice^{\thin,\epsilon} \right) \lesssim \frac{\pi^{N/2}}{N \cdot \Gamma(N/2)\zeta(N)}\epsilon^N,
    \end{equation*}
    and by Lemmas \ref{lemma:cuspvolumelowerbound2} and \ref{lemma:cuspvolumelowerbound3+}, we have
    \begin{equation*}
        \mu\left( \lattice^{\thin,\epsilon} \right) \ge \mu\left( \lattice^{\thin,\epsilon,1} \right) \gtrsim \frac{\pi^{N/2}}{N \cdot \Gamma(N/2)\zeta(N)}\epsilon^N.
    \end{equation*}
    Combining these inequalities yields the theorem.
\end{proof}

\begin{rmk}
    It is likely that one can show the correct leading asymptotic is $2$ in Theorem \ref{thm:cuspvolumequant}, i.e., that
    \begin{equation*}
        \mu\left( \lattice^{\thin,\epsilon} \right) = (2-o_{N \to \infty}(1)) \frac{\pi^{N/2}}{N\cdot\Gamma(N/2)\zeta(N)}\epsilon^N.
    \end{equation*}
    To see why, note that the map $\Lambda \mapsto \Lambda^\vee$ is a measure-preserving involution on $\lattice$ which sends $\lattice^{\thin,\epsilon,p}$ to $\lattice^{\thin,\epsilon,q}$, and we should expect
    \begin{equation*}
        \mu\left( \lattice^{\thin,\epsilon,p} \cap \lattice^{\thin,\epsilon,q} \right) \ll \mu\left( \lattice^{\thin,\epsilon,p}\right),
    \end{equation*}
    whenever $0 < \epsilon < 1$, since the systole of a random element of $\lattice$ is length $\asymp \sqrt{N} \gg 1$.
\end{rmk}

%% file: references.bib
@article{Cai,
  author    = {Cairo, H.},
  title     = {A Counterexample to the {M}izohata-{T}akeuchi Conjecture},
  journal   = {arXiv preprint arXiv:2502.06137},
  year      = {2025},
  month     = {mar},
  doi       = {10.48550/arXiv.2502.06137},
  url       = {https://arxiv.org/abs/2502.06137},
  shorthand = {Cai25}
}

@article{CaiZha,
  author    = {Cairo, H. and Zhang, R.},
  title     = {Power loss for the {M}izohata-{T}akeuchi conjecture on ${C}^k$ convex hypersurfaces},
  journal   = {arXiv preprint arXiv:2512.08064},
  year      = {2025},
  month     = {dec},
  doi       = {10.48550/arXiv.2512.08064},
  url       = {https://arxiv.org/abs/2512.08064},
  shorthand = {CZ25}
}

@article{CarIliWan,
  author    = {Carbery, A. and Iliopoulou, M. and Wang, H.},
  title     = {Some sharp inequalities of {M}izohata--{T}akeuchi-type},
  journal   = {Revista Matemática Iberoamericana},
  volume    = {40},
  number    = {4},
  pages     = {1387--1418},
  year      = {2024},
  month     = {feb},
  doi       = {10.4171/rmi/1463},
  url       = {https://ems.press/doi/10.4171/rmi/1463},
  issn      = {0213-2230, 2235-0616},
  shorthand = {CIW24},
}

@article{Sie,
shorthand = {Sie45},
  author    = {Siegel, C. L.},
  title     = {A Mean Value Theorem in Geometry of Numbers},
  journal   = {The Annals of Mathematics},
  volume    = {46},
  number    = {2},
  pages     = {340},
  year      = {1945},
  month     = {apr},
  doi       = {10.2307/1969027},
  url       = {https://www.jstor.org/stable/1969027},
  issn      = {0003486X}
}

@article{Rog,
shorthand = {Rog56},
  author    = {Rogers, C. A.},
  title     = {The Number of Lattice Points in a Set},
  journal   = {Proceedings of the London Mathematical Society},
  volume    = {s3-6},
  number    = {2},
  pages     = {305--320},
  year      = {1956},
  month     = {apr},
  doi       = {10.1112/plms/s3-6.2.305},
  url       = {http://doi.wiley.com/10.1112/plms/s3-6.2.305},
  issn      = {00246115},
  language  = {en}
}

@article{AthMar,
shorthand = {AM09},
  author    = {Athreya, J. and Margulis, G. A.},
  title     = {Logarithm laws for unipotent flows, {I}},
  journal   = {Journal of Modern Dynamics},
  volume    = {3},
  number    = {3},
  pages     = {359--378},
  year      = {2009}
}

@inproceedings{Ran,
shorthand = {Ran70},
  author    = {Randol, B.},
  title     = {A group-theoretic lattice-point problem},
  booktitle = {Problems in Analysis: A Symposium in Honor of Salomon Bochner},
  editor    = {Gunning, Robert C.},
  series    = {Princeton Mathematical Series},
  volume    = {31},
  pages     = {291--295},
  publisher = {Princeton University Press},
  address   = {Princeton, NJ},
  year      = {1970}
}

@inproceedings{Ste,
shorthand = {Ste79},
  author    = {Stein, E. M.},
  title     = {Some problems in harmonic analysis},
  booktitle = {Harmonic Analysis in Euclidean Spaces, Part 1},
  editor    = {Weiss, Guido and Wainger, Stephen},
  series    = {Proceedings of Symposia in Pure Mathematics},
  volume    = {35},
  pages     = {3--20},
  publisher = {American Mathematical Society},
  address   = {Providence, RI},
  year      = {1979}
}

@book{Lan,
shorthand = {Lan85},
  author    = {Lang, S.},
  title     = {$\mathrm{SL}_2(\mathbf{R})$},
  series    = {Graduate Texts in Mathematics},
  volume    = {105},
  publisher = {Springer},
  address   = {New York, NY},
  year      = {1985},
  doi       = {10.1007/978-1-4612-5142-2},
  url       = {https://link.springer.com/10.1007/978-1-4612-5142-2},
  isbn      = {9781461295815}
}

@article{Sch1,
shorthand = {Sch60},
  author    = {Schmidt, W. M.},
  title     = {A metrical theorem in geometry of numbers},
  journal   = {Transactions of the American Mathematical Society},
  volume    = {95},
  number    = {3},
  pages     = {516--529},
  year      = {1960},
  doi       = {10.1090/S0002-9947-1960-0117222-9},
  url       = {https://www.ams.org/tran/1960-095-03/S0002-9947-1960-0117222-9/},
  issn      = {0002-9947, 1088-6850},
  language  = {en}
}

@article{FaiHan,
shorthand = {FH25},
  author    = {Fairchild, S. and Han, J.},
  title     = {Mean value theorems for the ${S}$-arithmetic primitive {S}iegel transforms},
  journal   = {Journal of Modern Dynamics},
  volume    = {21},
  number    = {0},
  pages     = {645--692},
  year      = {2025},
  month     = {nov},
  doi       = {10.3934/jmd.2025015},
  url       = {https://www.aimsciences.org/en/article/doi/10.3934/jmd.2025015},
  issn      = {1930-5311},
  language  = {en}
}

@article{KelYu,
shorthand = {KY18},
  author    = {Kelmer, D. and Yu, S.},
  title     = {The second moment of the {S}iegel transform in the space of symplectic lattices},
  journal   = {arXiv preprint arXiv:1802.09645},
  year      = {2018},
  month     = {feb},
  doi       = {10.48550/arXiv.1802.09645},
  url       = {https://arxiv.org/abs/1802.09645}
}

@article{Vee,
shorthand = {Vee98},
  author    = {Veech, W. A.},
  title     = {Siegel Measures},
  journal   = {The Annals of Mathematics},
  volume    = {148},
  number    = {3},
  pages     = {895},
  year      = {1998},
  month     = {nov},
  doi       = {10.2307/121033},
  url       = {https://www.jstor.org/stable/121033},
  issn      = {0003486X}
}

@article{AthCheMas,
shorthand = {ACM19},
  author    = {Athreya, J. and Cheung, Y. and Masur, H.},
  title     = {Siegel-{V}eech transforms are in ${L}^2$},
  journal   = {arXiv preprint arXiv:1711.08537},
  year      = {2019},
  month     = {june},
  doi       = {10.48550/arXiv.1711.08537},
  url       = {https://arxiv.org/abs/1711.08537}
}

@article{Tak1,
shorthand = {Tak74},
  author    = {Takeuchi, J.},
  title     = {A necessary condition for the well-posedness of the {C}auchy problem for a certain class of evolution equations},
  journal   = {Proceedings of the Japan Academy},
  volume    = {50},
  number    = {2},
  pages     = {133--137},
  year      = {1974},
  publisher = {The Japan Academy}
}

@article{Tak2,
shorthand = {Tak80},
  author    = {Takeuchi, J.},
  title     = {On the {C}auchy problem for some non-{K}owalewskian equations with distinct characteristic roots},
  journal   = {Kyoto Journal of Mathematics},
  volume    = {20},
  number    = {1},
  year      = {1980},
  month     = {jan},
  doi       = {10.1215/kjm/1250522323},
  url       = {https://projecteuclid.org/journals/kyoto-journal-of-mathematics/volume-20/issue-1/On-the-Cauchy-problem-for-some-non-kowalewskian-equations-with/10.1215/kjm/1250522323.full},
  issn      = {2156-2261}
}

@book{Miz,
shorthand = {Miz85},
  author    = {Mizohata, S.},
  title     = {On the {C}auchy problem},
  series    = {Notes and Reports in Mathematics in Science and Engineering},
  volume    = {3},
  publisher = {Science Press and Academic Press},
  address   = {Beijing and Orlando, FL},
  year      = {1985}
}

@article{BenGutNakOli,
shorthand = {BGNO25},
  author    = {Bennett, J. and Gutiérrez, S. and Nakamura, S. and Oliveira, I.},
  title     = {A phase-space approach to weighted {F}ourier extension inequalities},
  journal   = {Forum of Mathematics, Sigma},
  volume    = {13},
  pages     = {e181},
  year      = {2025},
  month     = {jan},
  doi       = {10.1017/fms.2025.10127},
  url       = {https://www.cambridge.org/core/journals/forum-of-mathematics-sigma/article/phasespace-approach-to-weighted-fourier-extension-inequalities/AE5194721B554E6852CDB61655A370A9},
  issn      = {2050-5094},
  language  = {en}
}

@article{BenCarTao,
shorthand = {BCT06},
  author    = {Bennett, J. and Carbery, A. and Tao, T.},
  title     = {On the multilinear restriction and {K}akeya conjectures},
  journal   = {Acta Mathematica},
  volume    = {196},
  number    = {2},
  pages     = {261--302},
  year      = {2006},
  month     = {jan},
  doi       = {10.1007/s11511-006-0006-4},
  url       = {https://projecteuclid.org/journals/acta-mathematica/volume-196/issue-2/On-the-multilinear-restriction-and-Kakeya-conjectures/10.1007/s11511-006-0006-4.full},
  issn      = {0001-5962, 1871-2509},
  language  = {en}
}

@incollection{EskMoz,
shorthand = {EM22},
  title={Margulis functions and their applications},
  author={Eskin, A. and Mozes, S.},
  booktitle={Dynamics, Geometry, Number Theory: The Impact of Margulis on Modern Mathematics},
  editor={Fisher, David and Kleinbock, Dmitry and Soifer, Gregory},
  pages={342--361},
  year={2022},
  publisher={University of Chicago Press}
  }

@article{EskMirMoh,
shorthand = {EMM15},
title={Isolation, equidistribution, and orbit closures for the $\mathrm{SL}_2(\mathbb{R})$ action on moduli space}, ISSN={0003-486X}, url={https://annals.math.princeton.edu/2015/182-2/p07}, DOI={10.4007/annals.2015.182.2.7}, journal={Annals of Mathematics}, author={Eskin, A. and Mirzakhani, M. and Mohammadi, A.}, year={2015}, month=sept, pages={673–721}, language={en} }

@book{Cas,
shorthand = {Cas59},
  author    = {Cassels, J. W. S.},
  title     = {An Introduction to the Geometry of Numbers},
  publisher = {Springer},
  address   = {Berlin, Heidelberg},
  year      = {1959},
  doi       = {10.1007/978-3-642-62035-5},
  url       = {https://link.springer.com/10.1007/978-3-642-62035-5},
  isbn      = {9783540617884}
}

@article{IosSawSee,
shorthand = {ISS07},
  author    = {Iosevich, A. and Sawyer, E. T. and Seeger, A.},
  title     = {Mean lattice point discrepancy bounds, {II}: Convex domains in the plane},
  journal   = {Journal d’Analyse Mathématique},
  volume    = {101},
  number    = {1},
  pages     = {25--63},
  year      = {2007},
  month     = {mar},
  doi       = {10.1007/s11854-007-0002-4},
  url       = {https://doi.org/10.1007/s11854-007-0002-4},
  issn      = {1565-8538},
  language  = {en}
}

@article{BarRuiVeg,
shorthand = {BRV97},
  author    = {Barcelo, J. A. and Ruiz, A. and Vega, L.},
  title     = {Weighted Estimates for the {H}elmholtz Equation and Some Applications},
  journal   = {Journal of Functional Analysis},
  volume    = {150},
  pages     = {356--382},
  year      = {1997},
  month     = {nov},
  doi       = {10.1006/jfan.1997.3131},
  url       = {https://www.sciencedirect.com/science/article/pii/S0022123697931311},
  issn      = {0022-1236}
}

@article{Sch,
shorthand = {Sch58},
  author    = {Schmidt, W.},
  title     = {On the Convergence of Mean Values Over Lattices},
  journal   = {Canadian Journal of Mathematics},
  volume    = {10},
  pages     = {103--110},
  year      = {1958},
  month     = {jan},
  doi       = {10.4153/CJM-1958-013-2},
  url       = {https://www.cambridge.org/core/journals/canadian-journal-of-mathematics/article/on-the-convergence-of-mean-values-over-lattices/EC114326836609DCC60FFA8879B2FAAD},
  issn      = {0008-414X, 1496-4279},
  language  = {en}
}

@article{KleMar1,
shorthand = {KM99},
  author    = {Kleinbock, D. Y. and Margulis, G. A.},
  title     = {Logarithm laws for flows on homogeneous spaces},
  journal   = {Inventiones mathematicae},
  volume    = {138},
  number    = {3},
  pages     = {451--494},
  year      = {1999},
  month     = {dec},
  doi       = {10.1007/s002220050350},
  url       = {https://link.springer.com/10.1007/s002220050350},
  issn      = {0020-9910, 1432-1297},
  language  = {en}
}

@article{Hen,
shorthand = {Hen02},
  author    = {Henk, M.},
  title     = {Successive Minima and Lattice Points},
  journal   = {arXiv preprint arXiv:math/0204158},
  year      = {2002},
  month     = {apr},
  doi       = {10.48550/arXiv.math/0204158},
  url       = {https://arxiv.org/abs/math/0204158}
}

@article{Ber,
shorthand = {Ber12},
  author    = {Beresnevich, V.},
  title     = {Rational points near manifolds and metric Diophantine approximation},
  journal   = {Annals of Mathematics},
  volume    = {175},
  number    = {1},
  pages     = {187--235},
  year      = {2012},
  month     = {jan},
  doi       = {10.4007/annals.2012.175.1.5},
  url       = {http://annals.math.princeton.edu/2012/175-1/p05},
  issn      = {0003-486X},
  language  = {en}
}

@misc{CarLiPanYun,
shorthand = {CLPY25},
  author    = {Carbery, A. and Li, Z. K. and Pang, Y. and Yung, P-L},
  title     = {A weighted formulation of refined decoupling and inequalities of {M}izohata-{T}akeuchi-type for the moment curve},
  howpublished = {arXiv preprint arXiv:2510.04345},
  year      = {2025},
  month     = {oct},
  url       = {https://arxiv.org/abs/2510.04345},
  language  = {en}
}

@article{EinKatLin,
shorthand = {EKL06},
  author    = {Einsiedler, M. and Katok, A. and Lindenstrauss, E.},
  title     = {Invariant measures and the set of exceptions to {L}ittlewood’s conjecture},
  journal   = {Annals of Mathematics},
  volume    = {164},
  number    = {2},
  pages     = {513--560},
  year      = {2006},
  month     = {sept},
  doi       = {10.4007/annals.2006.164.513},
  url       = {http://annals.math.princeton.edu/2006/164-2/p04},
  issn      = {0003-486X},
  language  = {en}
}

@article{Hua,
shorthand = {Hua20},
  author    = {Huang, J.-J.},
  title     = {The density of rational points near hypersurfaces},
  journal   = {Duke Mathematical Journal},
  volume    = {169},
  number    = {11},
  year      = {2020},
  month     = {aug},
  doi       = {10.1215/00127094-2020-0004},
  url       = {https://projecteuclid.org/journals/duke-mathematical-journal/volume-169/issue-11/The-density-of-rational-points-near-hypersurfaces/10.1215/00127094-2020-0004.full},
  issn      = {0012-7094}
}

@article{KleMar,
shorthand = {KM98},
  author    = {Kleinbock, D. Y. and Margulis, G. A.},
  title     = {Flows on Homogeneous Spaces and Diophantine Approximation on Manifolds},
  journal   = {The Annals of Mathematics},
  volume    = {148},
  number    = {1},
  pages     = {339--360},
  year      = {1998},
  doi       = {10.2307/120997},
  url       = {https://www.jstor.org/stable/120997},
  issn      = {0003486X}
}

@article{SchSriTec,
shorthand = {SST23},
  author    = {Schindler, D. and Srivastava, R. and Technau, N.},
  title     = {Rational Points Near Manifolds, Homogeneous Dynamics, and Oscillatory Integrals},
  journal   = {arXiv preprint arXiv:2310.03867},
  year      = {2023},
  month     = {oct},
  doi       = {10.48550/arXiv.2310.03867},
  url       = {https://arxiv.org/abs/2310.03867}
}

@article{BerKleMar,
shorthand = {BKM02},
  author    = {Bernik, V. and Kleinbock, D. and Margulis, G. A.},
  title     = {Khintchine-type theorems on manifolds: the convergence case for standard and multiplicative versions},
  journal   = {Internat. Math. Res. Notices. 2001, no. 9, 453--486},
  year      = {2002},
  doi       = {10.48550/arxiv.math/0210298}
}

@article{EskMcM, 
shorthand = {EM93},
    title={Mixing, counting, and equidistribution in Lie groups}, 
    volume={71}, 
    ISSN={0012-7094, 1547-7398}, 
    url={https://projecteuclid.org/journals/duke-mathematical-journal/volume-71/issue-1/Mixing-counting-and-equidistribution-in-Lie-groups/10.1215/S0012-7094-93-07108-6.full}, 
    DOI={10.1215/S0012-7094-93-07108-6}, 
    number={1}, 
    journal={Duke Mathematical Journal}, 
    author={Eskin, A. and McMullen, C.}, 
    year={1993}, month=july, 
    pages={181–209}, 
    language={en} 
}

@book{Mor, 
shorthand = {Mor15},
series={Open textbook library}, title={Introduction to Arithmetic Groups}, ISBN={9780986571602}, abstractNote={Introduction to Arithmetic Groups}, publisher={Deductive Press}, author={Morris, D. W.}, year={2015}, collection={Open textbook library}, language={eng} }

@article{Doz, title={Measure bound for translation surfaces with short saddle connections}, volume={33}, ISSN={1420-8970}, url={https://doi.org/10.1007/s00039-023-00636-9}, DOI={10.1007/s00039-023-00636-9}, number={2}, journal={Geometric and Functional Analysis}, author={Dozier, Benjamin}, year={2023}, month=apr, pages={421–467}, language={en} }

@article{EskMarMoz, title={Upper Bounds and Asymptotics in a Quantitative Version of the Oppenheim Conjecture}, volume={147}, ISSN={0003486X}, url={https://www.jstor.org/stable/120984?origin=crossref}, DOI={10.2307/120984}, number={1}, journal={The Annals of Mathematics}, author={Eskin, A. and Margulis, G. A. and Mozes, S.}, year={1998}, month=jan, pages={93} }

@article{Smi,
shorthand = {Smi25},
title={Lattice points in thickened parabolas and rational points near hypersurfaces}, url={http://arxiv.org/abs/2512.00202}, DOI={10.48550/arXiv.2512.00202}, note={arXiv:2512.00202}, number={arXiv:2512.00202}, publisher={arXiv}, author={Smith, A.}, year={2025}, month=nov }

@article{Sri,
shorthand = {Sri25},
title={Counting rational points in non-isotropic neighborhoods of manifolds}, volume={478}, ISSN={0001-8708}, url={https://www.sciencedirect.com/science/article/pii/S0001870825002920}, DOI={10.1016/j.aim.2025.110394}, journal={Advances in Mathematics}, author={Srivastava, R.}, year={2025}, month=oct, pages={110394} }

@article{IosTay,
shorthand = {IT11},
title={Lattice points close to families of surfaces, non-isotropic dilations and regularity of generalized Radon transforms}, url={http://arxiv.org/abs/1103.1670}, DOI={10.48550/arXiv.1103.1670}, note={arXiv:1103.1670}, number={arXiv:1103.1670}, publisher={arXiv}, author={Iosevich, A. and Taylor, K.}, year={2011}, month=mar }

@article{Kim, title={Adelic Rogers integral formula}, url={http://arxiv.org/abs/2205.03138}, DOI={10.48550/arXiv.2205.03138}, note={arXiv:2205.03138}, number={arXiv:2205.03138}, publisher={arXiv}, author={Kim, S.}, year={2023}, month=aug }

@article{Pau, title={Comparison of volumes of {S}iegel sets and fundamental domains for $\mathrm{SL}_n(\mathbb{Z})$}, volume={199}, ISSN={1572-9168}, url={https://doi.org/10.1007/s10711-018-0350-5}, DOI={10.1007/s10711-018-0350-5}, number={1}, journal={Geometriae Dedicata}, author={Paula, G. T.}, year={2019}, month=apr, pages={291–306}, language={en} }

@book{Wei, address={Boston, MA}, title={Adeles and Algebraic Groups}, rights={http://www.springer.com/tdm}, ISBN={9781468491586}, url={http://link.springer.com/10.1007/978-1-4684-9156-2}, DOI={10.1007/978-1-4684-9156-2}, publisher={Birkhäuser Boston}, author={Weil, A.}, year={1982}, language={en} }

@online{Gar,
    title={Volumes of $\mathrm{SL}_n(\mathbb{Z})\backslash\mathrm{SL}_n(\mathbb{R})$ and $\mathrm{Sp}_n(\mathbb{Z})\backslash\mathrm{Sp}_n(\mathbb{R})$},
    year={2014},
    url={http://www-users.math.umn.edu/∼garrett/m/v/volumes.pdf},
    author={Garrett, P.},
    note={Accessed 05/22/2026}
}

@article{EinLinMicVen, title={The distribution of closed geodesics on the modular surface, and Duke’s theorem}, volume={58}, ISSN={0013-8584, 2309-4672}, url={https://ems.press/doi/10.4171/lem/58-3-2}, DOI={10.4171/lem/58-3-2}, number={3}, journal={L’Enseignement Mathématique}, author={Einsiedler, M. and Lindenstrauss, E. and Michel, P. and Venkatesh, A.}, year={2012}, month=dec, pages={249–313} }

@article{KadKleLinMar, title={Singular systems of linear forms and non-escape of mass in the space of lattices}, volume={133}, ISSN={0021-7670, 1565-8538}, url={http://link.springer.com/10.1007/s11854-017-0033-4}, DOI={10.1007/s11854-017-0033-4}, number={1}, journal={Journal d’Analyse Mathématique}, author={Kadyrov, S. and Kleinbock, D. and Lindenstrauss, E. and Margulis, G. A.}, year={2017}, month=oct, pages={253–277}, language={en} }

@article{Mor1, title={Bounding entropy for one-parameter diagonal flows on $\mathrm{SL}_{\mathtt{d}}(\mathbb{R})/\mathrm{SL}_{\mathtt{d}}(\mathbb{Z})$ using linear functionals}, ISSN={1435-9855}, url={https://ems.press/journals/jems/articles/14298574}, DOI={10.4171/jems/1592}, abstractNote={R. Mor}, journal={Journal of the European Mathematical Society}, author={Mor, Ron}, year={2025}, month=feb, language={en} }

@article{VDC, title={Verallgemeinerung einer Mordellschen Beweismethode in der Geometrie der Zahlen}, volume={1}, ISSN={0065-1036, 1730-6264}, url={http://www.impan.pl/get/doi/10.4064/aa-1-1-62-66}, DOI={10.4064/aa-1-1-62-66}, number={1}, journal={Acta Arithmetica}, author={Van Der Corput, J.}, year={1935}, pages={62–66}, language={en} }

@book{Zim, address={Boston, MA}, title={Ergodic Theory and Semisimple Groups}, rights={http://www.springer.com/tdm}, ISBN={9781468494907}, url={http://link.springer.com/10.1007/978-1-4684-9488-4}, DOI={10.1007/978-1-4684-9488-4}, publisher={Birkhäuser Boston}, author={Zimmer, Robert J.}, year={1984}, language={en} }

@article{EinKat, title={Invariant measures on $G/\Gamma$ for split simple Lie groups $G$}, volume={56}, ISSN={0010-3640, 1097-0312}, url={https://onlinelibrary.wiley.com/doi/10.1002/cpa.10092}, DOI={10.1002/cpa.10092}, number={8}, journal={Communications on Pure and Applied Mathematics}, author={Einsiedler, Manfred and Katok, Anatole}, year={2003}, month=aug, pages={1184–1221}, language={en} }
